\documentclass[11pt,a4paper]{article}

\usepackage[utf8]{inputenc}
\usepackage[T1]{fontenc}
\usepackage{lmodern}
\usepackage[a4paper,hmargin=2cm,vmargin=2.5cm]{geometry}
\usepackage{amsmath,amssymb,amsthm}
\usepackage{graphicx}
\usepackage{xcolor}
\usepackage{authblk}
\usepackage[numbers,sort&compress]{natbib}
\usepackage{url}

\theoremstyle{plain}
\newtheorem{theorem}{Theorem}
\newtheorem{proposition}[theorem]{Proposition}%
\newtheorem{lemma}[theorem]{Lemma}%
\newtheorem{corollary}[theorem]{Corollary}%

\theoremstyle{definition}

\newtheorem{notation}{Notation}
\newtheorem{assumption}{Assumption}

\theoremstyle{remark}
\newtheorem{remark}{Remark}%

\numberwithin{equation}{section}

\newcommand{\dx}{\mathrm{d}x}
\newcommand{\dsigma}{\mathrm{d}\sigma}
\newcommand{\dt}{\mathrm{d}t}
\newcommand{\ds}{\mathrm{d}s}
\newcommand{\dz}{\mathrm{d}z}

\usepackage{tikz}
\usepackage{pgfplots}
\pgfplotsset{compat=newest}
\usepackage{pgfplotstable}

\pgfplotstableread[col sep=space]{
h L2 H1
2.828427124746190624e-01 7.865312283988015740e-02 1.478836433459725652e-01
1.414213562373095312e-01 1.387503777921226741e-02 6.744318689623447771e-02
7.071067811865476560e-02 2.597917294477363435e-03 3.275274325472957160e-02
3.535533905932738280e-02 5.494773674174813520e-04 1.612695555393450775e-02
1.767766952966369140e-02 1.231557624150977795e-04 8.017874087636796612e-03
}\FigFourPOneData

\pgfplotstableread[col sep=space]{
h L2 H1
2.828427124746190624e-01 1.233446254963804951e-02 6.698834886043765480e-02
1.414213562373095312e-01 3.235327331867929920e-03 3.253961957419163747e-02
7.071067811865476560e-02 7.925792181431726494e-04 1.610015519049503088e-02
3.535533905932738280e-02 1.921869832114108940e-04 8.019310788168806242e-03
1.767766952966369140e-02 4.698456969783042388e-05 4.002250163936154313e-03
}\FigFourPTwoData

\pgfplotstableread[col sep=space]{
h_squared L2(FEM) H1(FEM) L2 H1
1.989731058892788840e-01 1.079017751957941050e-02 4.377726308769035873e-02 1.285378946497648944e-01 2.850969168725091540e-01
9.994638746611982094e-02 2.671797028785069441e-03 1.900147886656795848e-02 2.645291901094792841e-02 1.375146942726353239e-01
4.999843394275343272e-02 6.412155548806956372e-04 8.823519920479555578e-03 5.539355217659703776e-03 6.668629088113349479e-02
2.499764073201517223e-02 1.612196638008814656e-04 4.515687839386813528e-03 1.260136343556173552e-03 3.200851544596220594e-02
1.249997469885046238e-02 3.983765176546507616e-05 2.173323691464657675e-03 2.800975259495349940e-04 1.423671648430036539e-02
}\FigFourFEMData

\pgfplotstableread[col sep=space]{
gamma h rel_L2_H1 rel_Linf_L2
1.000000000000000056e-01 2.000000000000000111e-01 7.442660687658178009e-01 5.101600142934795690e-01
1.000000000000000056e-01 1.000000000000000056e-01 2.070863525139429728e-01 7.918143908398406206e-02
1.000000000000000056e-01 5.000000000000000278e-02 7.842277813087447647e-02 1.245319451615151986e-02
1.000000000000000056e-01 2.500000000000000139e-02 3.162782318910882517e-02 2.449391906008903299e-03
1.000000000000000000e+00 2.000000000000000111e-01 2.694751965810245342e-01 1.101938490367220386e-01
1.000000000000000000e+00 1.000000000000000056e-01 1.244985476076275049e-01 1.985613487151873316e-02
1.000000000000000000e+00 5.000000000000000278e-02 5.685440587513872213e-02 3.980579656022571917e-03
1.000000000000000000e+00 2.500000000000000139e-02 2.616305659583396948e-02 9.210415403783423806e-04
1.000000000000000000e+01 2.000000000000000111e-01 3.863122175644372325e-01 3.370617028180838415e-01
1.000000000000000000e+01 1.000000000000000056e-01 1.375783072140974173e-01 6.517264092544171328e-02
1.000000000000000000e+01 5.000000000000000278e-02 5.554203182167614244e-02 9.163014292991954191e-03
1.000000000000000000e+01 2.500000000000000139e-02 2.561895599221934022e-02 1.151530061064877115e-03
1.000000000000000000e+02 2.000000000000000111e-01 9.562628793847091080e-01 9.945081788351353147e-01
1.000000000000000000e+02 1.000000000000000056e-01 4.622992116022908160e-01 3.459339026660267491e-01
1.000000000000000000e+02 5.000000000000000278e-02 1.511549054951883686e-01 7.555350931018708915e-02
1.000000000000000000e+02 2.500000000000000139e-02 4.110327460113701037e-02 1.161063053174286180e-02
1.000000000000000000e+03 2.000000000000000111e-01 1.209019034776908086e+00 1.251673049020526785e+00
1.000000000000000000e+03 1.000000000000000056e-01 8.560713597662927832e-01 6.390674856233025114e-01
1.000000000000000000e+03 5.000000000000000278e-02 4.920822544533455245e-01 2.561409811193491315e-01
1.000000000000000000e+03 2.500000000000000139e-02 2.007400684702384808e-01 7.044364006088439656e-02
}\FigFiveData

\pgfplotstableread[col sep=space]{
h L2 H1
1.414213562373095312e-01 8.565770165878981018e-03 7.677646348854842706e-02
7.071067811865476560e-02 2.060627635183127221e-03 3.763319546545370459e-02
3.535533905932738280e-02 5.012847012052984224e-04 1.868381184743373585e-02
}\FigSixData

\pgfplotstableread[col sep=space]{
h L2 H1
7.071067811865476560e-02 1.424536830005208128e-03 7.050114272659968075e-03
3.535533905932738280e-02 1.529158659183590930e-04 2.659043026567096436e-03
1.767766952966369140e-02 2.588325325604790380e-05 1.228623980063562292e-03
8.838834764831845700e-03 5.629368529171168762e-06 6.027355029313033420e-04
4.419417382415922850e-03 1.351142395091094178e-06 3.000094156093401358e-04
}\FigSevenData

\pgfplotstableread[col sep=space]{
h L2 H1
1.154700538379251490e-01 1.720175021909841478e-02 1.053955015036648762e-01
8.660254037844386521e-02 9.331032802809106907e-03 7.880643245321712620e-02
6.928203230275509217e-02 5.748623911324872784e-03 5.889962105328153119e-02
5.773502691896257449e-02 3.900555879897314292e-03 5.250594044913367564e-02
}\FigEightData

\newcommand{\logLogSlopeTriangle}[5]
{

    \pgfplotsextra
    {
        \pgfkeysgetvalue{/pgfplots/xmin}{\xmin}
        \pgfkeysgetvalue{/pgfplots/xmax}{\xmax}
        \pgfkeysgetvalue{/pgfplots/ymin}{\ymin}
        \pgfkeysgetvalue{/pgfplots/ymax}{\ymax}

        \pgfmathsetmacro{\xArel}{#1}
        \pgfmathsetmacro{\yArel}{#3}
        \pgfmathsetmacro{\xBrel}{#1-#2}
        \pgfmathsetmacro{\yBrel}{\yArel}
        \pgfmathsetmacro{\xCrel}{\xArel}

        \pgfmathsetmacro{\lnxB}{\xmin*(1-(#1-#2))+\xmax*(#1-#2)} 
        \pgfmathsetmacro{\lnxA}{\xmin*(1-#1)+\xmax*#1} 
        \pgfmathsetmacro{\lnyA}{\ymin*(1-#3)+\ymax*#3} 
        \pgfmathsetmacro{\lnyC}{\lnyA+#4*(\lnxA-\lnxB)}
        \pgfmathsetmacro{\yCrel}{(\lnyC-\ymin)/(\ymax-\ymin)} 

        \coordinate (A) at (rel axis cs:\xArel,\yArel);
        \coordinate (B) at (rel axis cs:\xBrel,\yBrel);
        \coordinate (C) at (rel axis cs:\xCrel,\yCrel);

        \draw[#5]   (A)-- node[pos=0.5,anchor=north] {1}
                    (B)--
                    (C)-- node[pos=0.5,anchor=west] {#4}
                    cycle;
    }
}

\usepackage[colorlinks=true,linkcolor=blue,citecolor=blue,urlcolor=blue]{hyperref}

\begin{document}

\title{A \texorpdfstring{$\varphi$}{phi}-FEM approach for time-dependent domains with unfitted meshes}

\author[1]{Michel Duprez\thanks{Corresponding author: \href{mailto:michel.duprez@inria.fr}{michel.duprez@inria.fr}}}
\author[2]{Johan Marguet}
\author[2]{Alexei Lozinski}
\author[3]{Igor Voulis}
\affil[1]{ICube, Université de Strasbourg, CNRS, Inria, F-67000 Strasbourg, France}
\affil[2]{LMB (UMR 6623), Université Marie et Louis Pasteur, CNRS, F-25000 Besançon, France}
\affil[3]{Institute for Numerical and Applied Mathematics, University of Göttingen, Lotzestr. 16--18, 37083 Göttingen, Germany}
\date{}

\maketitle

\begin{abstract}
In this work, we propose an unfitted finite element scheme to approximate the
solution of the heat equation on moving domains. We use the $\varphi$-FEM paradigm in
which the computational domain is described implicitly by a level-set function
$\varphi$. This function is incorporated into the variational formulation in order to
enforce the boundary conditions, which allows the use of unfitted meshes in space.
Such a strategy avoids the need for remeshing at each time step and makes it possible
to handle complex geometrical evolutions. Moreover, the $\varphi$-FEM approach has the
advantage of being simple to implement within standard finite element libraries. We
introduce a fully discrete scheme that combines the $\varphi$-FEM spatial
discretization with the lowest-order discontinuous Galerkin method in time. Under
regularity assumptions on the level-set function and a mild restriction linking the
time step to the mesh size, we establish an optimal \emph{a priori} error estimate in
the $L^2(0,T;H^1)$ norm. Finally, we present several numerical experiments that
confirm this convergence rate, exhibit a second-order convergence in the
$L^{\infty}(0,T;L^2)$ norm, and illustrate the robustness and accuracy of the
proposed method.


\end{abstract}

\section{Introduction}

The classical finite element method (FEM) typically relies on meshes that conform to the boundary of the computational domain. While this requirement is natural for simple geometries, it can become restrictive when dealing with complex shapes or with domains that evolve in time.
To overcome this difficulty, Immersed Boundary Methods (IBM) \cite{IBMrev} have been developed. These approaches allow the use of background meshes that are independent of the physical boundary by extending the governing equations to a larger computational domain.
Recently, the CutFEM methodology \cite{cutfem, cutfem_mixed} has emerged as a robust alternative. In this framework, the variational formulation is integrated only over the portions of the mesh elements intersected by the physical domain, while additional stabilization terms are introduced to guarantee numerical stability and optimal accuracy.
Another related approach is the Shifted Boundary Method \cite{sbm}, which approximates the boundary conditions through a Taylor expansion around the true boundary.

Most unfitted finite element methods have originally been developed for partial differential equations posed on static domains. These methods provide appropriate frameworks for problems on complex geometries since they rely on a background mesh that does not need to conform to the domain boundary, thereby avoiding costly mesh generation or mesh deformation procedures. However, when the computational domain evolves in time, additional difficulties arise due to the changing intersection between the moving boundary and the background mesh. In particular, as the interface moves across the mesh, new elements may appear while others disappear, which requires dynamically updating the geometric reconstruction, the quadrature rules, and possibly the stabilization terms. In addition, the time discretization becomes more delicate, since expressions such as $\partial_t u \approx \tau^{-1}(u^n-u^{n-1})$ are not straightforward to define when $u^n$ and $u^{n-1}$ are defined on different domains.
As a consequence, extending classical unfitted approaches to moving domains leads to additional implementation and analysis challenges. Nevertheless, several recent works have proposed cutFEM discretizations specifically designed to handle evolving geometries; see for instance \cite{heimann2025discretization,heimann2023geometrically,preuss2018higher}
 and \cite{burman2022eulerian,wahl-richter-lehrenfeld2022unfitted} for Eulerian time-stepping schemes.

In this work, we adopt the $\varphi$-FEM approach, originally introduced in \cite{duprez2020phi} for the Poisson equation with Dirichlet boundary conditions. The main idea of $\varphi$-FEM is to incorporate the level-set function $\varphi$ directly into the discrete formulation. In this framework, boundary conditions can be enforced either through penalization or by enriching the finite element space with the level-set function $\varphi$. As a consequence, the computational mesh does not need to conform to the boundary of the domain.
Additional stabilization terms are introduced to ensure optimal convergence rates as well as an optimal conditioning of the resulting finite element matrix. One of the main advantages of the method lies in its ease of implementation.
Several extensions of the method have been proposed in the literature, including formulations for Neumann boundary conditions \cite{neumann}, the Stokes equations \cite{duprez2023phi}, as well as crack and interface problems \cite{chapitre}. The $\varphi$-FEM approach has also been combined with neural network techniques to accelerate computations \cite{phiFEMFNO}, and adapted to finite difference discretizations \cite{phiFD}. The case of the heat equation on a stationary domain has been studied in \cite{heat}. In the present work, we extend this analysis to the case of moving boundaries.

The manuscript is organized as follows. In Section~\ref{Sect:Scheme_and_main_results}, we introduce the numerical scheme and state the main theoretical results. Their proofs are provided in Section~\ref{Sect:Proof}. Finally, Section~\ref{Sect:Numerical_results} presents several numerical experiments illustrating the method and confirming the theoretical convergence results.

\section{\texorpdfstring{{$\varphi$}}{phi}-FEM scheme and main results}\label{Sect:Scheme_and_main_results}

In this article, we consider the heat equation in the moving domain $\Omega (t)\subset \mathbb{R} ^d$ ($d=2,3$) for $t\in [0,T]$ with $T>0$:
\begin{subequations}\label{eq:heat}
\begin{alignat}{3}
	\frac{\partial u}{\partial t} - \Delta u &= f && \quad \text{in } \Omega (t)  && \text{ for all } t \in [0, T], \\
	u &= 0 && \quad \text{on } \partial \Omega (t) && \text{ for all } t \in [0, T], \\
	u(0) &= u^0 && \quad \text{on } \Omega (0), &&
\end{alignat}
\end{subequations}
where {$f\in L^2(Q_T)$, $Q_T = \left\{ (t, x) : t \in [0, T] \text{ and } x \in \Omega (t) \right\}$}
and $u^0\in L^2(\Omega(0))$. We refer to \cite{burman2022eulerian} for the well-posedness of this {problem (that reference actually deals with a more complicated Stokes system, but its arguments apply verbatim to the simpler situation considered here).}

We suppose that the domain is described by a given level set function $\varphi$, i.e. for all $t\in[0,T]$
\[ \Omega (t) = \{x \in \mathbb{R}^d : \varphi (t, x) < 0\}. \]
For the discretization in time, we introduce the uniform grid $t_n = n \tau$ with
$n = 0, 1, \ldots, N$ ($N\in\mathbb{N}^*)$ and $\tau=T/N$ the time step. We also denote by $I_n = [t_{n -
	1}, t_n]$ the time slab. For the discretization in space, we introduce first the
background mesh $\mathcal{T}_h^{\mathcal{O}}$ on a domain $\mathcal{O}$ that
contains  $\Omega (t)$ {for all $t\in[0,T]$.} For each $n \geqslant 1$, we define
the active mesh for the slab $I_n$
\[ \mathcal{T}_h^n = \{T \in \mathcal{T}_h^{\mathcal{O}} : T \cap \Omega (t)\neq \varnothing \text{ for some } t \in I_n \}.
\]
We shall also denote by $\Omega_h^n$ the domain occupied by mesh $\mathcal{T}_h^n$, i.e. the interior of $\cup_{T\in\mathcal{T}_h^n}T$. {Note that $\Omega(t)\subset\Omega_h^n$ for all $t\in I_n$.}

As a first step towards the derivation of our $\varphi$-FEM scheme,
we suppose (on a formal level) that $u (t, \cdot)$ can be extended
from $\Omega (t)$ to $\Omega_h^n$ {for all $t\in I_n$ on each time slab $I_n$}, and this extension (still denoted by $u$) satisfies the heat equation on $I_n \times \Omega_h^n$ with the RHS (still denoted by $f$) which should be also extended to $I_n \times \Omega_h^n$. We then multiply the governing PDE by a test function $v = v (t, x)$ and integrate over $I_n \times \Omega_h^n$. This gives, {writing henceforth $u(t_n)$ instead of $u(t_n,\cdot)$ whenever there is no ambiguity,}
\begin{multline*} \int_{\Omega_h^n} u (t_n) v (t_n) \dx - \int_{I_n} \int_{\Omega_h^n} u
\frac{\partial v}{\partial t} \dx \dt + \int_{I_n}
\int_{\Omega_h^n} \nabla u \cdot \nabla v \dx \dt - \int_{I_n}
\int_{\partial \Omega_h^n} \partial_{\nu}{u} v \dsigma
\dt \\
 = \int_{\Omega_h^n} u (t_{n - 1}) v (t_{n - 1}) \dx + \int_{I_n} \int_{\Omega_h^n} fv \dx \dt.
\end{multline*}
Now, taking the lowest-order (piecewise constant in time) discontinuous Galerkin approximation, we approximate $u(t, x)$ by $\varphi (t, x)  \tilde{u}^n (x)$ for each $t \in I_n$ and $\tilde{u}^n$ defined on $\Omega_h^n$, so that the homogeneous Dirichlet boundary conditions on $\partial\Omega(t)$ are automatically satisfied. Similarly, the test function is taken as $v (t, x) = \varphi (t, x)\tilde{v}^n(x)$ for $t \in I_n$ and $\tilde{v}^n$ defined on $\Omega_h^n$. The weak formulation above suggests then the following discretization in time:
\begin{multline*}	
	\int_{\Omega_h^n} \varphi^2 (t_n)  \tilde{u}^n  \tilde{v}^n
	\dx - \int_{I_n} \int_{\Omega_h^n} \frac{\partial \varphi}{\partial t}  \tilde{u}^n \varphi  \tilde{v}^n \dx \dt + \int_{I_n} \int_{\Omega_h^n}
	\nabla (\varphi \tilde{u}^n) \cdot \nabla (\varphi \tilde{v}^n) \dx \dt -
	\int_{I_n} \int_{\partial \Omega_h^n} \partial_{\nu}{(\varphi\tilde{u}^n)}  \varphi \tilde{v}^n  \dsigma \dt
\\
  = \int_{\Omega_h^n \cap \Omega_h^{n - 1}} \varphi^2 (t_{n - 1})  \tilde{u}^{n
	- 1}  \tilde{v}^n \dx + \int_{\Omega_h^n \setminus \Omega_h^{n - 1}} \varphi^2 (t_{n - 1})  \tilde{u}^n  \tilde{v}^n \dx + \int_{I_n}
\int_{\Omega_h^n} f \varphi \tilde{v}^n \dx \dt.
\end{multline*}
We have made here the following choice to approximate the integral
$\int_{\Omega_h^n} u (t_{n - 1}) v \dx$: we approximate $u (t_{n - 1})$ by the available information from the previous
time slab, i.e. $\varphi (t_{n - 1})  \tilde{u}^{n - 1}$ wherever possible, i.e. on $\Omega_h^n \cap \Omega_h^{n - 1}$; on the remaining part $\Omega_h^n\setminus \Omega_h^{n - 1}$, we approximate $u (t_{n - 1})$ by $\varphi (t_{n - 1})  \tilde{u}^n$.

To discretize in space, we introduce the standard {$\mathbb{P}_k$} finite
element spaces with a fixed  $k \geqslant 1$ for each time slab $I_n$, $n = 1, \ldots, N$:
\[ V_h^n =\left\{v \in H^1 (\Omega^n_h) : v|_T \in \mathbb{P}_k (T),   \ \forall T \in \mathcal{T}^n_h \right\} . \]
Using these spaces to discretize in space the trial and test function in the formulation above, and computing explicitly the integrals of $\varphi$ in time when possible, {and adding some stabilization terms to be specified below,} we now state our $\varphi$-FEM scheme as: for each $n = 1, \ldots,
N$, find $\tilde{u}_h^n \in V_h^n$ s.t.
\begin{multline}\label{Eq:scheme}
  \int_{\Omega_h^n \cap \Omega_h^{n - 1}} \frac{\varphi^2 (t_n) + \varphi^2 (t_{n
  - 1})}{2}  \tilde{u}_h^n  \tilde{v}_h^n \dx + \int_{\Omega_h^n
  \setminus \Omega_h^{n - 1}} \frac{\varphi^2 (t_n) - \varphi^2 (t_{n - 1})}{2}
  \tilde{u}_h^n  \tilde{v}_h^n \dx \\
  + \int_{I_n} \int_{\Omega_h^n} \nabla (\varphi \tilde{u}^n_h) \cdot \nabla
  (\varphi \tilde{v}_h^n) \dx \dt - \int_{I_n} \int_{\partial
  \Omega_h^n} \partial_{\nu}{(\varphi \tilde{u}^n_h)} \varphi  \tilde{v}_h^n
  \dsigma \dt + \ \text{stab} (\tilde{u}_h^n,
  \tilde{v}_h^n)  \\
   = \int_{\Omega_h^n \cap \Omega_h^{n - 1}} \varphi^2 (t_{n - 1})
  \tilde{u}_h^{n - 1}  \tilde{v}_h^n \dx + \int_{I_n} \int_{\Omega_h^n}
  f \varphi \tilde{v}_h^n \dx \dt + \hspace{0.17em}
  \text{stab}_{\text{RHS}} (\tilde{v}_h^n)
\end{multline}
for all $\tilde{v}_h^n \in V_h^n$. At the initial step, $n = 1$, the first term on the RHS should be replaced by
\[ \int_{\Omega_h^0} u^0 \varphi (t_0)  \tilde{v}_h^1 \dx, \]
where $\Omega_h^0$ is the domain of the mesh at initial time
\[ \mathcal{T}_h^0 = \{T \in \mathcal{T}_h^{\mathcal{O}} : T \cap \Omega (0)
   \neq \varnothing\} , \]
{and} $u^0$ is the given initial condition, which should be extended from
$\Omega (0)$ to $\Omega_h^0$. In this way, we avoid the explicit introduction of $\tilde{u}_h^0$, bypassing the division of the initial condition by $\varphi (0)$, which could be cumbersome in practice. 

Stabilization terms, as inspired directly by the $\varphi$-FEM for
Poisson-Dirichlet \cite{duprez2020phi} and for the heat equation \cite{heat}, are
\begin{multline*} \text{stab} (\tilde{u}_h^n, \tilde{v}_h^n) = \gamma h \sum_{F\in\mathcal{F}_h^{n, \Gamma}}  \int_{I_n}\int_{F}
 [\![\partial_\nu (\varphi \tilde{u}^n_h) ]\!]
   [\![\partial_\nu (\varphi \tilde{v}_h^n) ]\!] \dsigma \dt \\
   + \gamma h^2 {\sum_{T\in\mathcal{T}_h^{n, \Gamma}}} \int_{I_n} \int_{T} \left(
   \frac{\partial \varphi}{\partial t}  \tilde{u}^n_h - \Delta (\varphi
   \tilde{u}^n_h) \right)  \left( \frac{\partial \varphi}{\partial t}
   \tilde{v}^n_h - \Delta (\varphi \tilde{v}_h^n) \right) \dx \dt
\end{multline*}
and, on the RHS of the scheme,
\[ \text{stab}_{\text{RHS}} (\tilde{v}_h^n) = \gamma h^2  {\sum_{T\in\mathcal{T}_h^{n, \Gamma}}} \int_{I_n}\int_{T} f \left( \frac{\partial \varphi}{\partial t}
   \tilde{v}^n_h - \Delta (\varphi \tilde{v}_h^n) \right) \dx \dt. \]
Here, $\mathcal{T}_h^{n, \Gamma}$ is the submesh of $\mathcal{T}_h^n$ consisting of the cells visited by the boundary $\partial \Omega (t)$ during the time interval $I_n$. More precisely,
\[ \mathcal{T}_h^{n, \Gamma} = \{ T \in \mathcal{T}_h^n : \exists~ t \in I_n \text{ s.t. }T \not\subset \Omega (t) \}. \]
{We denote by $\Omega_h^{n,\Gamma}$ the domain occupied by $\mathcal{T}_h^{n,\Gamma}$, i.e. the interior of $\cup_{T\in\mathcal{T}_h^{n,\Gamma}}T$; note that $\Omega_h^n\setminus\Omega(t)\subset\Omega_h^{n,\Gamma}$ for every $t\in I_n$, and $\Omega_h^n\setminus\Omega_h^{n-1}\subset\Omega_h^{n,\Gamma}$.}
Moreover, $\mathcal{F}_h^{n, \Gamma}$ is the set of the interior mesh faces of $\mathcal{T}_h^n$ belonging to $\mathcal{T}_h^{n, \Gamma}$, i.e. $\mathcal{F}_h^{n, \Gamma}$ collects the interior faces of the mesh $\mathcal{T}_h^{n, \Gamma}$ plus the faces on the boundary of $\mathcal{T}_h^{n, \Gamma}$ which are not on the boundary of $\mathcal{T}_h^{n}$. {On all such faces, $[\![\cdot ]\!]$ stand for the jumps.}

\begin{remark}\label{remNonConsis}
The time derivative approximation $\frac{\partial\varphi}{\partial t}  \tilde{u}^n_h$ in the stabilization term is not fully consistent with the exact derivative $\frac{\partial u}{\partial t}=\frac{\partial \varphi}{\partial t}\tilde{u}^n+\varphi \frac{\partial \tilde{u}^n}{\partial t}$. It lacks indeed a discretization of $\varphi \frac{\partial \tilde{u}}{\partial t}$. We shall see
however, that these inconsistencies do not hinder the order of convergence of the scheme.
\end{remark}

The analysis of the method will be carried out assuming, essentially, that the level set $\varphi$ is sufficiently smooth, and that it changes sufficiently slowly in time, so that the boundary $\Gamma(t)=\{x:\varphi(t,x)=0\}$ crosses a small number of cell layers when time goes from $t_{n-1}$ to $t_n$. Moreover, $\varphi$ should behave like the signed distance to $\Gamma(t)$ near $\Gamma(t)$ (without requiring that it is exactly the signed distance). More precisely, we shall adopt the following assumptions:

\begin{assumption}\label{asm0}
	At any time $t\in[0,T]$, the boundary $\Gamma(t)$ can be covered by open sets $\mathcal{U}_i$, $i=1,\ldots,N_{\mathcal{U}}=N_{\mathcal{U}}(t)$, and one can introduce on every $\mathcal{U}_i$ local coordinates $\xi_1,\ldots,\xi_d$ with $\xi_d=\varphi(t,\cdot)$ such that all the partial derivatives $\partial^\alpha\xi/\partial x^\alpha$ and $\partial^\alpha x/\partial \xi^\alpha$ up to order $k+1$ are bounded by some $C_{\xi}>0$, uniformly in $t$ and $i$. 
	{Moreover, there exist a positive integer $N_\text{over}$ and $\delta,m>0$, all independent of $t$ and $i$, such that
	\begin{itemize}
		\item[(i)] in the local coordinates $\xi_1,\ldots,\xi_d$ on $\mathcal{U}_i$, this set is a cylinder $\omega_i\times(-\delta,\delta)$;  
		\item[(ii)] 
		each $\mathcal{U}_i$ has a non-empty intersection with at most $N_\text{over}-1$ other $\mathcal{U}_j$, $j\not =i$;
		\item[(iii)] $|\varphi(t,\cdot)|\ge m$ on $\mathcal{O}\setminus\cup_{i=1,\ldots,N_{\mathcal{U}}}\mathcal{U}_i$ with some $m>0$;
		\item[(iv)] for every $n=1,\ldots,N$ and every $t\in I_n$,  $\Omega_h^{n,\Gamma}\subset\cup_{i=1,\ldots,N_{\mathcal{U}}}\mathcal{U}_i$.   
	\end{itemize}}
\end{assumption}

\begin{assumption}\label{asm1}
	The level set  $\varphi(t,x)$ is of class $C^1$ in $t$ and of class ${C^{k+1}}$ in $x$ on $[0,T]\times \mathcal{O}$. All the corresponding derivatives of $\varphi$ are bounded by some $C_{\varphi}>0$.
\end{assumption}

\begin{assumption}\label{asm2}
	For every $n\in\{1,\ldots,N\}$, the cells of the boundary submesh $\mathcal{T}_h^{n,\Gamma}$ can be {distributed among} patches $\{\Pi_i \}_{i = 1,\ldots, N_{\Pi}}$, {each patch being made of some cells of $\mathcal{T}_h^{n,\Gamma}$ together with one additional cell taken in the interior,} having the following properties (cf. Fig. \ref{FigPatch}):
	\begin{itemize}
		\item Each patch $\Pi_i$ is a connected set composed of a mesh cell $T_i\in\mathcal{T}_h^n\setminus\mathcal{T}_h^{n,\Gamma}$ and some mesh elements in $\mathcal{T}_h^{n,\Gamma}$. More precisely, $\Pi_i = \{T_i\} \cup \Pi_i^{\Gamma}$ with $\Pi_i^{\Gamma}{}\subset\mathcal{T}_h^{n,\Gamma}$ containing at most $M_\Pi$ cells;
		\item $\mathcal{T}_h^{n,\Gamma} = \cup_{i = 1}^{N_{\Pi}} \Pi_i^{\Gamma}$;
		\item $\Pi_i \cap \Pi_j = \varnothing$ if $i \neq j$.
	\end{itemize}
	By a slight abuse of notation, we shall also write $\Pi_i$ and $\Pi_i^\Gamma$ for the corresponding open domains, i.e. the interiors of $\cup_{T\in\Pi_i}T$ and of $\cup_{T\in\Pi_i^\Gamma}T$. The domains $\Pi_i$ are connected.
\end{assumption}

The assumptions above are essentially the same as those used in the stationary case in \cite{duprez2020phi}. The purpose of Assumption \ref{asm0} is to guarantee the Hardy inequality of Lemma \ref{LemHardy} below with a constant independent of $t$, $n$ and $h$.
Assumptions \ref{asm1} and \ref{asm2} guarantee that $\mathcal{T}_h^{n,\Gamma}$ remains a band of width $O(h)$ around $\Gamma(t)$ for every $t\in I_n$, a property used repeatedly in the analysis below (Lemmas \ref{LemHardy}--\ref{LemMag1}).

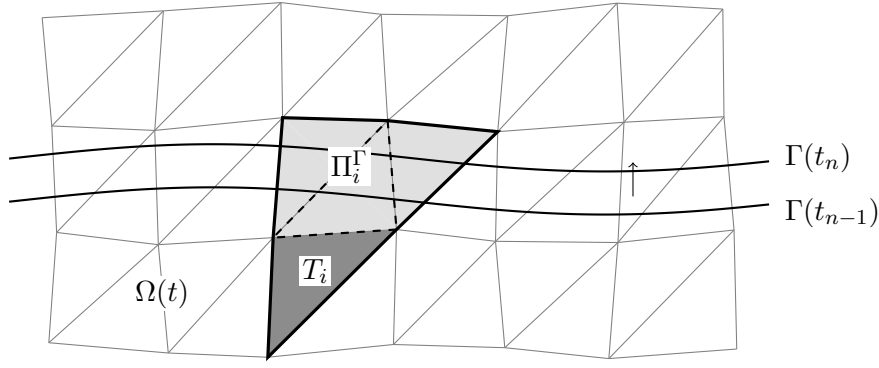
\begin{figure}[tbp]
\centering
\begin{tikzpicture}[scale=1.5]
	\foreach \i in {0,...,6}{
		\foreach \j in {0,...,3}{
			\pgfmathsetmacro{\xx}{\i+0.07*sin(131*\i+79*\j)}
			\pgfmathsetmacro{\yy}{\j+0.07*cos(107*\i+53*\j)}
			\coordinate (P\i-\j) at (\xx,\yy);
		}
	}
	\fill[black!12] (P2-1)--(P3-1)--(P2-2)--cycle;
	\fill[black!12] (P3-1)--(P3-2)--(P2-2)--cycle;
	\fill[black!12] (P3-1)--(P4-2)--(P3-2)--cycle;
	\fill[black!45] (P2-0)--(P3-1)--(P2-1)--cycle;
	\foreach \i/\ii in {0/1,1/2,2/3,3/4,4/5,5/6}{
		\foreach \j in {0,...,3}{ \draw[gray,thin] (P\i-\j)--(P\ii-\j); }
	}
	\foreach \i in {0,...,6}{
		\foreach \j/\jj in {0/1,1/2,2/3}{ \draw[gray,thin] (P\i-\j)--(P\i-\jj); }
	}
	\foreach \i/\ii in {0/1,1/2,2/3,3/4,4/5,5/6}{
		\foreach \j/\jj in {0/1,1/2,2/3}{ \draw[gray,thin] (P\i-\j)--(P\ii-\jj); }
	}
	\draw[very thick] (P2-0)--(P3-1)--(P4-2)--(P3-2)--(P2-2)--(P2-1)--cycle;
	\draw[thick,dashed] (P3-1)--(P2-1)--(P3-2)--cycle;
	\draw[thick,smooth,variable=\x,domain=-0.35:6.35,samples=80]
	     plot ({\x},{1.32+0.12*sin(52*\x+15)});
	\draw[thick,smooth,variable=\x,domain=-0.35:6.35,samples=80]
	     plot ({\x},{1.70+0.12*sin(52*\x+15)});
	\draw[->] (5.15,1.36)--(5.15,1.66);
	\node[anchor=west] at (6.4,1.22) {$\Gamma(t_{n-1})$};
	\node[anchor=west] at (6.4,1.70) {$\Gamma(t_n)$};
	\node[fill=white,inner sep=1pt] at (1.0,0.5) {$\Omega(t)$};
	\node[fill=white,inner sep=1pt] at (2.65,1.62) {$\Pi_i^\Gamma$};
	\node[fill=white,inner sep=1pt] at (2.35,0.70) {$T_i$};
\end{tikzpicture}
\caption{A sketch of a patch $\Pi_i$ from Assumption \ref{asm2}, in the situation where the boundary moves from $\Gamma(t_{n-1})$ to $\Gamma(t_n)$ during the time slab $I_n$. Light grey: the cells of $\Pi_i^\Gamma \subset \mathcal{T}_h^{n,\Gamma}$, i.e. cells visited by the boundary during $I_n$; dark grey: the cell $T_i\in\mathcal{T}_h^n\setminus\mathcal{T}_h^{n,\Gamma}$, so that $T_i\subset\Omega(t)$ for every $t\in I_n$; and so $\Pi_i = \{T_i\} \cup \Pi_i^\Gamma$. The dashed segments are the interior faces $\mathcal{F}_i$ of the patch, cf. the proof of Lemma \ref{LemMag1}.}\label{FigPatch}
\end{figure}

\begin{notation} 
Let $Q_T$ denote the space-time domain where the exact problem is posed, i.e.
\[ Q_T = \left\{ (t, x) : t \in [0, T] \text{ and } x \in \Omega (t) \right\} \]
and $Q_n$ denote the space-time slabs corresponding to the time intervals
$I_n$, $n = 1, \ldots, N$, i.e.
\[ Q_n = \left\{ (t, x) : t \in I_n \text{ and }x \in \Omega (t) \right\}. \]
Let $\mathcal{H} ^k (Q_T)$ be the following Bochner type space of function on
$Q_T$:
\begin{multline}\label{Hkspace} \mathcal{H} ^k (Q_T) =
	\bigg\{ v \in L^2 (Q_T) : v (t,
	\cdot) \in H^{k + 1} ( \Omega (t)), \\
	\frac{\partial v}{\partial t} (t, \cdot)
	\in H^k ( \Omega (t)), \
	\frac{\partial^2 v}{\partial t^2} (t, \cdot) \in H^{k	- 1} ( \Omega (t)) \text{ for a.a. }t \in [0, T] \bigg\}
\end{multline}
and define the norm on this space accordingly
\[ \| v \|_{\mathcal{H} ^k (Q_T)}^2 = \int_{0}^T \left( \| v (t, \cdot)
\|_{H^{k + 1} ( \Omega (t))}^2 + \left\| \frac{\partial v}{\partial t} (t,
\cdot) \right\|_{H^k ( \Omega (t))}^2 + \left\| \frac{\partial^2
	v}{\partial t^2} (t, \cdot) \right\|_{H^{k - 1} ( \Omega (t))}^2 \right)
\dt. \]
The spaces $\mathcal{H} ^k (Q_n)$ are defined similarly for $n = 1, \ldots,
N$. {In the proofs of Section \ref{Sect:Proof}, we shall abbreviate the norm of $L^2(I\times D)$ by $\|\cdot\|_{I\times D}$ for a measurable set $D$ and an interval $I$, the norm of $L^2(D)$ by $\|\cdot\|_{D}$, the seminorm of $H^s(D)$ by $|\cdot|_{s,D}$.}
\end{notation}

\begin{theorem}
	\label{thm:error}Adopting Assumptions \ref{asm0}--\ref{asm2},
	let $u \in \mathcal{H} ^k (Q_T)$ be the exact solution to \eqref{eq:heat},
	and suppose that the RHS $f (t, x)$ is extended for every $t \in I_n$, $n =
	1, \ldots, N$, to the bigger domain $\Omega^n_h$, so that $f |_{I_n \times \Omega_h^n}  \in L^2 (I_n ; H^{k - 1} (\Omega^n_h))$, and the
	initial condition $u^0$ is extended to $\Omega^0_h$, so that $u^0 \in H^{k +1} (\Omega^0_h)$. Let \ $u_h^n = \varphi \tilde{u}^n_h$ be the discrete
	solution given by \ \eqref{Eq:scheme}. For $\gamma$ large enough, there
	exist $c, C, h_0, \tau_0 > 0$ depending only on the regularity of the mesh $\mathcal{T}_h^{\mathcal{O}}$, on $k$, on $\gamma$
	and on the constants in Assumptions \ref{asm0}--\ref{asm2},
	such that if $h\leqslant h_0$ and $ch^2\leqslant\tau\leqslant\tau_0$ then
	\begin{multline*}
		\left( \sum_{n = 1}^N \int_{I_n} |u  (t, \cdot) - u_h^n (t, \cdot) |_{H^1
			(\Omega (t))}^2 \dt \right)^{\frac{1}{2}} \leqslant Ch^k \|u^0 \|_{H^{k + 1}
			(\Omega_h^0)}\\
		+ C (h^k + \tau)  \left( \|u\|_{\mathcal{H} ^k (Q_T)} + \left( \sum_{n =
			1}^N \| f \|_{L^2 (I_n ; H^{k - 1} (\Omega^n_h))}^2 \right)^{\frac{1}{2}}
		\right).
	\end{multline*}
\end{theorem}

\begin{remark}
The $L^{\infty}(L^2)$ counterpart of this estimate is not proved here. By analogy with the case of the heat exation in a static domain, treated in \cite{heat}, we conjecture that 
\begin{multline*} \max_{1 \leqslant n \leqslant N} \ \sup_{t \in I_n} \|u(t,\cdot) - u_h^n(t,\cdot)\|_{L^2(\Omega(t))}
\lesssim h^{k+1}\|u^0\|_{H^{k+1}(\Omega_h^0)} \\
+ (h^{k+{1/2}} + \tau) \left( \|u\|^2_{\mathcal{H}^k(Q_T)} + \sum_{n=1}^N \|f\|_{L^2(I_n;H^{k-1}(\Omega_h^n))}^2 \right)^{1/2}.
\end{multline*}
The numerical experiments of Section \ref{Sect:Numerical_results} give evidence of the second-order convergence in $h$ for $k=1$, similar to observations in \cite{heat}.
\end{remark}

\section{Proof of Theorem \ref{thm:error}}\label{Sect:Proof}

In this section, we will first prove the coercivity of the bilinear form, which will be combined in the next subsection with the consistency to obtain the convergence results claimed in Theorem \ref{thm:error}.

\subsection{Coercivity of the bilinear form}

{Let $\mathbb{V}_h$ denote the space of functions in space-time combining the contributions of FE spaces $V^n_h$ multiplied by the level set $\varphi$, i.e. $\mathbb{V}_h=\{U_h = (\varphi \tilde{u}^n_h)_{n \in \{1, \ldots, N\}},\text{ with } \tilde{u}^n_h\in V^n_h\}$.  We shall also use the shortcuts
$u^n_h = \varphi \tilde{u}^n_h$. Similarly, the notation $V_h\in\mathbb{V}_h$ implies $V_h = (\varphi \tilde{v}^n_h)_{n \in \{1, \ldots, N\}}$ and $v^n_h = \varphi\tilde{v}^n_h$.}
With these conventions, scheme (\ref{Eq:scheme}) can be written as: find $U_h\in\mathbb{V}_h$ s.t.
\[ a_h (U_h, V_h) = F (V_h) \quad \forall V_h\in\mathbb{V}_h \]
with
\begin{multline}
  \label{ahalt} a_h (U_h, V_h) = \sum_{n = 1}^N \left[ \int_{\Omega_h^n \cap
  \Omega_h^{n - 1}} \varphi^2 (t_{n - 1}) (\tilde{u}_h^n - \tilde{u}_h^{n - 1}
  \mathbf{1}_{n > 1}) \tilde{v}_h^n \dx + \int_{\Omega_h^n}
  \frac{\varphi^2 (t_n) - \varphi^2 (t_{n - 1})}{2}  \tilde{u}_h^n  \tilde{v}_h^n
  \dx \right.
\\
 + \int_{I_n} \int_{\Omega_h^n} \nabla u^n_h \cdot \nabla v^n_h \dx
   \dt - \int_{I_n} \int_{\partial \Omega_h^n} \partial_{\nu}{u^n_h} v_h^n \dsigma \dt \\
\left. + \gamma h \int_{I_n}\sum_{F\in\mathcal{F}_h^{n, \Gamma}}  \int_{F}
   [\![\partial_\nu u^n_h]\!] [\![\partial_\nu v_h^n]\!] \dsigma \dt +
   \gamma h^2  \sum_{T\in\mathcal{T}_h^{n,\Gamma}}\int_{I_n} \int_{T} \left(
   \frac{\partial \varphi}{\partial t}  \tilde{u}^n_h - \Delta u^n_h \right)
   \left( \frac{\partial \varphi}{\partial t}  \tilde{v}^n_h - \Delta v_h^n
   \right) \dx \dt \right]
\end{multline}
and
\[ F (V_h) = \sum_{n = 1}^N \left[ \int_{I_n} \int_{\Omega_h^n} fv_h^n \dx \dt + \sum_{T\in\mathcal{T}_h^{n,\Gamma}} \gamma h^2  \int_{I_n} \int_{T} f \left( \frac{\partial \varphi}{\partial t}  \tilde{v}^n_h - \Delta
   v_h^n \right) \dx \dt \right] + \int_{\Omega_h^0} u^0 \varphi (t_0)
   \tilde{v}_h^1 \dx. \]

Put $V_h = U_h$ and remark that the first line in the definition of $a_h (U_h,
U_h)$ can be rewritten as the sum of the contributions inside the
corresponding physical domains $\Omega (t_n)$ and $\Omega (t_{n - 1})$
\begin{multline}  \label{termsIN}
	\mathbb{T}^{\text{in}}:=
	\sum_{n = 1}^N \left[ \int_{\Omega (t_{n - 1})} \varphi^2 (t_{n
  - 1}) (\tilde{u}_h^n - \tilde{u}_h^{n - 1} \mathbf{1}_{n > 1})
  \tilde{u}_h^n \dx \right.\\
  \left. + \int_{\Omega (t_n)} \frac{\varphi^2 (t_n)}{2} | \tilde{u}_h^n |^2
  \dx - \int_{\Omega (t_{n - 1})} \frac{\varphi^2 (t_{n - 1})}{2} |
  \tilde{u}_h^n |^2 \dx \right],
\end{multline}
and the remainder
\begin{multline}  \label{termsOUT}
  \mathbb{T}^{\text{out}}:=
	\sum_{n = 1}^N \left[ \int_{(\Omega_h^n \cap \Omega_h^{n -
  1}) \setminus \Omega (t_{n - 1})} \varphi^2 (t_{n - 1}) (\tilde{u}_h^n -
  \tilde{u}_h^{n - 1} \mathbf{1}_{n > 1}) \tilde{u}_h^n \dx \right.\\
  \left. + \int_{\Omega_h^n \setminus \Omega (t_n)} \frac{\varphi^2 (t_n)}{2} |
  \tilde{u}_h^n |^2 \dx - \int_{\Omega_h^n \setminus \Omega (t_{n - 1})}
  \frac{\varphi^2 (t_{n - 1})}{2} | \tilde{u}_h^n |^2 \dx \right] .
\end{multline}

We shall refer to (\ref{termsIN}) as the terms ``in'', and to \
(\ref{termsOUT}) as the terms ``out''. The terms ``in'' can be simplified
using the identity $(\tilde{u}_h^n - \tilde{u}_h^{n - 1})  \tilde{u}_h^n =
\frac{1}{2}  (| \tilde{u}_h^n |^2 - | \tilde{u}_h^{n - 1} |^2 + |
\tilde{u}_h^n - \tilde{u}_h^{n - 1} |^2)$, as follows:
\begin{multline} \label{eq:new in}
\mathbb{T}^{\text{in}}=
\sum_{n = 2}^N \left[ \int_{\Omega (t_n)} \frac{\varphi^2 (t_n)}{2} |
   \tilde{u}_h^n |^2 \dx - \int_{\Omega (t_{n - 1})} \frac{\varphi^2 (t_{n -
   1})}{2} | \tilde{u}_h^{n - 1} |^2 \dx + \int_{\Omega (t_{n - 1})}
   \frac{\varphi^2 (t_{n - 1})}{2} | \tilde{u}_h^n - \tilde{u}_h^{n - 1} |^2
   \dx \right] \\
 + \int_{\Omega (t_0)} \frac{\varphi^2 (t_0)}{2} | \tilde{u}_h^1 |^2 \dx +
   \int_{\Omega (t_1)} \frac{\varphi^2 (t_1)}{2} | \tilde{u}_h^1 |^2 \dx \\
= \int_{\Omega (t_N)} \frac{\varphi^2 (t_N)}{2} | \tilde{u}_h^N |^2 \dx +
   \int_{\Omega (t_0)} \frac{\varphi^2 (t_0)}{2} | \tilde{u}_h^1 |^2 \dx +
   \sum_{n = 2}^N \int_{\Omega (t_{n - 1})} \frac{\varphi^2 (t_{n - 1})}{2}  |
   \tilde{u}_h^n - \tilde{u}_h^{n - 1} |^2 \dx.
   \end{multline}
The terms ``out'' can be rewritten, denoting $B^n (t) = \Omega_h^n \setminus\Omega (t)$, as
\begin{multline} \label{eq:new out}
\mathbb{T}^{\text{out}}=
\sum_{n = 2}^N \int_{\Omega_h^n \cap B^{n - 1} (t_{n - 1})} \varphi^2 (t_{n -
   1})  (\tilde{u}_h^n - \tilde{u}_h^{n - 1})  \tilde{u}_h^n \dx +
   \int_{B^0 (t_0)} \varphi^2 (t_0) | \tilde{u}_h^1 |^2 \dx \\
 + \sum_{n = 1}^N \int_{I_n} \int_{B^n (t)} \varphi \frac{\partial
   \varphi}{\partial t} | \tilde{u}_h^n |^2 \dx \dt.
\end{multline}
Indeed, the last two terms in (\ref{termsOUT}) are
$\frac{1}{2}\int_{B^n(t_n)}\varphi^2(t_n)|\tilde u^n_h|^2\dx-\frac{1}{2}\int_{B^n(t_{n-1})}\varphi^2(t_{n-1})|\tilde u^n_h|^2\dx=\frac{1}{2}\int_{I_n} \frac{\mathrm{d}}{\dt}\int_{B^n (t)} \varphi^2(t)  | \tilde{u}_h^n |^2 \dx \dt$; the Reynolds transport theorem applies and the boundary contribution coming from the motion of $\Gamma(t)$ vanishes since $\varphi(t,\cdot)=0$ on $\Gamma(t)$. We have also used $\Omega_h^0\subset\Omega_h^1$ to write the term with $n=1$ in the first line.

This allows us to decompose $a_h (U_h, U_h)$ into three parts
\begin{equation}
  \label{3termsIOS} a_h (U_h, U_h) = a_h^{\text{in}} (U_h, U_h) +
  a_h^{\text{out}} (U_h, U_h) + a_h^{\text{stab}} (U_h, U_h) .
\end{equation}
Here, the ``in'' contribution $a_h^{\text{in}} (U_h, U_h)$ regroups the parts
of the integrals on $\Omega (t)$:
\[ a_h^{\text{in}} (U_h, U_h) = \mathbb{T}^{\text{in}}
+ \sum_{n = 1}^N \int_{I_n} \int_{\Omega
   (t)} | \nabla u^n_h |^2 \dx \dt \]
In view of \eqref{eq:new in}, they can be written as
\begin{multline*}
a_h^{\text{in}} (U_h, U_h) = \frac{1}{2}  \int_{\Omega (t_N)} \varphi^2 (t_N)
   | \tilde{u}_h^N |^2 \dx + \frac{1}{2}  \int_{\Omega (t_0)} \varphi^2
   (t_0) | \tilde{u}_h^1 |^2 \dx + \sum_{n = 1}^N \int_{I_n} \int_{\Omega
   (t)} | \nabla u^n_h |^2 \dx \dt \\
 + \sum_{n = 2}^N \frac{1}{2}  \int_{\Omega (t_{n - 1})} \varphi^2 (t_{n - 1})
   | \tilde{u}_h^n - \tilde{u}_h^{n - 1} |^2 \dx.
   \end{multline*}

The ``out'' contribution $a_h^{\text{out}} (U_h, U_h)$ regroups the parts of
the integrals outside $\Omega (t)$, i.e. on $B^n (t)$, and the boundary
integrals on $\partial \Omega_h^n$: \
\[ a_h^{\text{out}} (U_h, U_h) = \mathbb{T}^{\text{out}}
+ \sum_{n = 1}^N \Bigg[ \int_{I_n} \int_{B ^n (t)} |\nabla u^n_h|^2 \dx \dt
- \int_{I_n} \int_{\partial\Omega_h^n} \partial_{\nu}{u^n_h} \, u^n_h \dsigma \dt \Bigg].
\]
Using \eqref{eq:new out} and integrating by parts in $x$ on $T \cap B^n (t)$
for each cell $T \in\mathcal{T}^{n, \Gamma}_h$ gives
\[ a_h^{\text{out}} (U_h, U_h) = \sum_{n = 1}^N \left[ \sum_{T\in\mathcal{T}_h^{n,\Gamma}}\int_{I_n}
   \int_{T \setminus \Omega (t)} \left( \frac{\partial
   \varphi}{\partial t}  \tilde{u}_h^n - \Delta u^n_h \right) u^n_h \dx
   \dt \right. \left. + \sum_{F\in\mathcal{F}_h^{n, \Gamma}} \int_{I_n}  \int_{F \cap
   B^n (t)} [\![\partial_\nu u^n_h]\!] u^n_h \dsigma \dt \right] \]
\[ + \sum_{n = 2}^N \int_{\Omega^n_h \cap B^{n - 1} (t_{n - 1})} \varphi^2 (t_{n- 1})  (\tilde{u}_h^n - \tilde{u}_h^{n - 1})  \tilde{u}_h^n \dx
+\int_{B^0 (t_0)} \varphi^2 (t_0) | \tilde{u}_h^1 |^2 \dx. \]

Finally, the stabilization terms read
\[ a_h^{\text{stab}} (U_h, U_h) = \sum_{n = 1}^N \left[ \ \gamma
   h \sum_{F\in\mathcal{F}_h^{n, \Gamma}} \int_{I_n}  \int_{F}  [\![\partial_\nu u^n_h]\!]^2 \dsigma \dt + \gamma h^2  \sum_{T\in\mathcal{T}_h^{n,\Gamma}} \int_{I_n} \int_{T} \left| \frac{\partial \varphi}{\partial t}  \tilde{u}^n_h - \Delta u^n_h \right|^2 \dx \dt \right]. \]

\

With the aid of Young inequality, the expressions for $a_h^{\text{in}}$ and
$a_h^{\text{out}}$ lead to the following bound from below, valid for any
$\varepsilon > 0$:
\begin{align}
  \label{BIGbound} a_h^{\text{in}} (U_h, U_h) & + a_h^{\text{out}} (U_h, U_h)
   \geqslant \frac{1}{2}  \int_{\Omega (t_N)} \varphi^2 (t_N) | \tilde{u}_h^N |^2
  \dx + \frac{1}{2}  \int_{\Omega_h^0} \varphi^2 (t_0) | \tilde{u}_h^1 |^2
  \dx
\\
 & + \sum_{n = 1}^N \underbrace{\int_{I_n} \int_{\Omega (t)} | \nabla u^n_h
   |^2 \dx \dt}_{\text{cf. Lemma } \ref{LemMag1}} +
   \sum_{n = 2}^N \frac{1}{2}  \int_{\Omega (t_{n - 1})} \varphi^2 (t_{n - 1})  |
   \tilde{u}_h^n - \tilde{u}_h^{n - 1} |^2 \dx
 \notag   \\
 & - \sum_{n = 1}^N \left[ \frac{h^2}{2 \varepsilon}  \underbrace{\sum_{T\in\mathcal{T}_h^{n,\Gamma}}\int_{I_n}
   \int_{T} \left| \frac{\partial \varphi}{\partial t}
   \tilde{u}_h^n - \Delta u^n_h \right|^2 \dx \dt}_{\text{ part of stabilization}} + \frac{\varepsilon}{2 h^2}
   \underbrace{\sum_{T\in\mathcal{T}_h^{n,\Gamma}}\int_{I_n} \int_{T} |u^n_h |^2 \dx
   \dt}_{\text{ cf.  Lemma } \ref{LemInverse}} \right]
 \notag \\
 & - \sum_{n = 1}^N \left[ \frac{h}{2 \varepsilon}  \underbrace{
  \sum_{F\in\mathcal{F}_h^{n, \Gamma}} \int_{I_n} \int_{F} [\![\partial_\nu u^n_h]\!]^2 \dsigma \dt}_{\text{ part of stabilization}} +
   \frac{\varepsilon}{2 h} \underbrace{ \sum_{F\in\mathcal{F}_h^{n, \Gamma}}\int_{I_n}   \int_{F}|u^n_h |^2 \dsigma \dt}_{\text{cf. Lemma }
   \ref{LemInverse}}  \right]
  \notag \\
    &+ \sum_{n = 2}^N \left[ \frac{1}{2}  \int_{\Omega^n_h \cap B^{n - 1} (t_{n - 1})} \varphi^2 (t_{n - 1}) | \tilde{u}_h^n |^2 \dx \right.
   \left. - \frac{1}{2}  \underbrace{\int_{\Omega^n_h \cap B^{n - 1}  (t_{n - 1})} \varphi^2 (t_{n - 1}) | \tilde{u}_h^{n - 1} |^2 \dx }_{\text{cf. Lemma } \ref{LemInverseT}} \right].
\notag
\end{align}

This bound contains several negative terms. Some of them are part of
stabilization which should be yet added to the expression above, and
compensate them directly. The others will be treated, as indicated, by
upcoming Lemmas \ref{LemInverse} and \ref{LemInverseT}, bounding them by a sum
of the positive terms, present in the scheme and the stabilization. Note also
that the first underbraced term, although positive, controls the $H^1$ norm
only on $\Omega (t)$, while we shall need it on the larger domains
$\Omega^n_h$. This term will be treated by Lemma \ref{LemMag1}, as indicated.

We now turn to the proofs of the Lemmas announced above. We begin with the Hardy inequality, which is the basic tool behind Lemmas \ref{LemInverse}, \ref{LemInverseT}.  Its generalization to higher orders is also crucial for the interpolation theory of the next section. 

\begin{lemma}\label{LemHardy}
{Under Assumptions \ref{asm0}--\ref{asm1}, there exists $C_H>0$ such that, for all $n=1,\ldots,N$, all $t\in I_n$ and all $v\in H^1(\Omega_h^n)$ vanishing on $\Gamma(t)$,}
\begin{equation}\label{Hardy}
	{\left\| \frac{v}{\varphi(t,\cdot)} \right\|_{\Omega_h^n} \leqslant C_H |v|_{1,\Omega_h^n}  .}
\end{equation}
\end{lemma}

\begin{proof}
Let $\mathcal{U}=\cup_i \mathcal{U}_i$ be the neighborhood of $\Gamma(t)$ from Assumption \ref{asm0} and $v\in H^1(\mathcal{U}\cup\Omega^n_h)$ be a function vanishing on $\Gamma(t)$. Adopting local coordinates $\xi_1,\ldots,\xi_d$ on each $\mathcal{U}_i$ and remarking that $v$ vanishes at $\xi_d=0$, we can apply the usual Hardy inequality to bound the $L^2(\mathcal{U}_i)$ norm of $v/\xi_d$ by the $H^1(\mathcal{U}_i)$ semi-norm of $v$. We then go back to original coordinates, recall that $\xi_d=\varphi(t,\cdot)$, and sum over the domains $\mathcal{U}_i$:
$$
\left\| \frac{v}{\varphi(t,\cdot)} \right\|_{0,\mathcal{U}}^2
\le \, \sum_{i=1}^{N_{\mathcal{U}}} \left\| \frac{v}{\varphi(t,\cdot)} \right\|_{0,\mathcal{U}_i}^2
\lesssim  \, \sum_{i=1}^{N_{\mathcal{U}}} |v|_{1,\mathcal{U}_i}^2
\le  \, N_{\text{over}} |v|_{1,\mathcal{U}}^2
$$
using the assumption that any point $x\in\mathcal{U}$ belongs to at most $N_{\text{over}}$ sets $\mathcal{U}_i$.

On $\Omega_h^n\setminus\mathcal{U}$, one has $|\varphi(t,\cdot)|\geqslant m$, so that $$\left\| \frac{v}{\varphi(t,\cdot)} \right\|_{\Omega_h^n\setminus\mathcal{U}}\leqslant \frac1m\|v\|_{\Omega_h^n\setminus\mathcal{U}}\leqslant \frac{C_P}{m}|v|_{1,\Omega(t)}$$ 
by Poincaré-Friedrichs inequality on $\Omega(t)$ valid for functions vanishing on $\Gamma(t)$,  with a constant $C_P$ depending only on the diameter of   $\Omega(t)$. The last inequality uses that $\Omega_h^n\setminus\mathcal{U}\subset\Omega(t)$, which follows from Assumption \ref{asm0}(iv): a cell of $\mathcal{T}_h^n$ meeting $\Omega_h^n\setminus\mathcal{U}$ cannot belong to $\mathcal{T}_h^{n,\Gamma}\subset\mathcal{U}$ and is therefore contained in $\Omega(t)$. Note that the global bound $\|v\|_{\Omega_h^n}\lesssim|v|_{1,\Omega_h^n}$ cannot be used at this stage, since it is a consequence of the present lemma, cf. Corollary \ref{LemPoin}.
Combining with the estimate on $\mathcal{U}$, we conclude $\|v/\varphi(t,\cdot)\|_{\Omega_h^n}\lesssim |v|_{1,\mathcal{U}\cup\Omega_h^n}$. This entails (\ref{Hardy}) since a function in $H^1(\Omega^n_h)$ can be extended on $\mathcal{U}\cup\Omega^n_h$ with continuous dependence of $H^1$ norms. Note that the constants $N_\text{over}$, $m$, etc are independent of $t$, which makes $C_H$ independent of $t$, $n$ and $h$.
\end{proof}

\begin{corollary}\label{LemPoin}
	{Under Assumptions \ref{asm0}--\ref{asm1}, there exists $C_p>0$ such that, for all $n=1,\ldots,N$, all $t\in I_n$ and all $v\in H^1(\Omega_h^n)$ vanishing on $\Gamma(t)$,}
	\begin{equation}\label{Poin}
			\|v\|_{\Omega_h^{n,\Gamma}} \leqslant C_p h\, |v|_{1,\Omega_h^n}  
			\qquad\text{and}\qquad
			\|v\|_{\Omega_h^n} \leqslant C_p |v|_{1,\Omega_h^n}.
	\end{equation}
\end{corollary}

\begin{proof}
Indeed, by (\ref{Hardy}) and the fact that $\varphi(t,\cdot)$ is of order $h$ on $\Omega_h^{n,\Gamma}$,
$$
\|v\|_{\Omega_h^{n,\Gamma}}   \le \left\| \varphi(t,\cdot)\frac{v}{\varphi(t,\cdot)} \right\|_{\Omega_h^{n,\Gamma}} \lesssim h\left\| \frac{v}{\varphi(t,\cdot)} \right\|_{\Omega_h^{n,\Gamma}} \leqslant C_H h\, |v|_{1,\Omega_h^n}\,.
$$
The second inequality in (\ref{Poin}) follows by  Poincaré-Friedrichs inequality on $\Omega(t)$ since $\Omega_h^{n}\subset\Omega_h^{n,\Gamma}\cup\Omega(t)$ for any $t\in I_n$.
\end{proof}

\begin{lemma}
  \label{LemInverse}Under Assumptions {\ref{asm0}--\ref{asm2}}, there exists $C>0$ such that, for each $n=1,\ldots,N$, {all $h\leqslant h_0$} and $v_h \in  V_h^n$,
  \begin{equation}
    \frac{1}{h^2}  \sum_{T\in\mathcal{T}_h^{n,\Gamma}}\int_{I_n} \int_{T} | \varphi v_h |^2
    \dx \dt + \frac{1}{h} \sum_{F\in\mathcal{F}_h^{n, \Gamma}}   \int_{I_n} \int_{F} | \varphi v_h |^2 \dsigma \dt \leqslant C \int_{I_n}
    \int_{\Omega^n_h} | \nabla (\varphi v_h) |^2 \dx \dt.
    \label{evident1}
  \end{equation}
\end{lemma}

\begin{proof}  {Thanks to Assumption \ref{asm2},} $\Omega_h^{n, \Gamma}$ is a band of  width of order $h$ around $\Gamma (t)$ for every $t\in I_n$. Since $\varphi(t) v_h$ vanishes on
	$\Gamma(t)$, Poincar{\'e} inequality (\ref{Poin}) gives for every $t\in I_n$,
	\[ \| \varphi(t) v_h \|_{0, \Omega_h^{n, \Gamma}} \lesssim h\, | \varphi(t) v_h |_{1,\Omega_h^{n}}. \]
	
  By the standard scaled trace inequality {on each cell $T\in\mathcal{T}_h^{n,\Gamma}$, i.e. $\|w\|^2_{0,\partial T}\lesssim \frac1h\|w\|^2_{F}+h|w|^2_{1,T}$ for $w\in H^1(T)$}, we also have
  \[ \sum_{F\in\mathcal{F}_h^{n, \Gamma}}  \| \varphi(t) v_h \|_{0, F}^2 \lesssim \frac{1}{h} \|
     \varphi(t) v_h \|_{0, \Omega_h^{n, \Gamma}}^2 + h | \varphi(t) v_h |^2_{1,\Omega_h^{n, \Gamma}}. \]
  Integrating these in time on $I_n$ leads to (\ref{evident1}).
\end{proof}

\begin{lemma}
  \label{LemInverseT}Under Assumptions {\ref{asm0}--\ref{asm2}}, there exists $C>0$ such that, for each $n=1,\ldots,N$, {and all $v_h \in  V_h^n$,  $s\in[t_{n-1},t_n]$,}
\begin{equation}
    {\int_{\Omega_h^{n,\Gamma}}} \varphi^2 (s) |v_h |^2 \dx
    \leqslant C {\left(\frac{h^2}{\tau}+\tau\right)}  \int_{I_n} \int_{\Omega^n_h} | \nabla (\varphi v_h) |^2 \dx \dt \,. \label{evident2}
  \end{equation}
\end{lemma}

\begin{proof}
		We start by the equality $\varphi (s, x) = \varphi (t, x) + \int_{t}^s \frac{\partial \varphi}{\partial t} (z, x) \dz$, valid for all $t\in I_n$ and $x \in \Omega_h^{n, \Gamma}$. Taking the square, multiplying by $| v_h |^2 (x)$ and integrating in $x$ and $t$, leads to
		\[ \tau \int_{\Omega_h^{n, \Gamma}} \varphi^2 (s) |v_h |^2 \dx
		\leqslant 2 \int_{I_n} \int_{\Omega^{n, \Gamma}_h} | \varphi v_h |^2 \dx \dt + \tau^2 \int_{I_n} \int_{\Omega^{n, \Gamma}_h} \left|
		\frac{\partial \varphi}{\partial t} v_h \right|^2 \dx \dt \,.\]
		The first integral on the right-hand side above is already bounded by $h^2
		\int_{I_n} \int_{\Omega^n_h} | \nabla (\varphi v_h) |^2 \dx \dt$ thanks to
		Lemma \ref{LemInverse}. To bound the second integral, we recall that $\frac{\partial\varphi}{\partial t}$ is uniformly bounded by Assumption~\ref{asm1} and conclude
		\[ \int_{I_n} \int_{\Omega^{n, \Gamma}_h} \left| \frac{\partial
			\varphi}{\partial t} v_h \right|^2 \dx \dt \lesssim \int_{I_n}
		\int_{\Omega^{n, \Gamma}_h} \left|  \frac{\varphi v_h}{\varphi} \right|^2
		\dx \dt \lesssim \int_{I_n} \int_{\Omega^n_h} | \nabla (\varphi v_h) |^2
		\dx \dt \,.\]
		The latter is valid by the Hardy inequality of Lemma \ref{LemHardy}.
		Putting these estimates together and dividing by $\tau$ gives the announced
		result.
\end{proof}

\begin{lemma}
	\label{LemMag1}
	Under Assumptions~\ref{asm0}--\ref{asm2}, there exist constants $C,h_0>0$, depending only on the regularity of the mesh $\mathcal{T}_h^{\mathcal{O}}$, on $k$ and on the constants of Assumptions~\ref{asm0}--\ref{asm2}, such that for all $h\leqslant h_0$, all $n=1,\ldots,N$ and all $v_h \in V_h^n$,
	\begin{multline*}
		\int_{I_n} \int_{\Omega_h^n} | \nabla (\varphi v_h) |^2 \, \dx \, \dt
		\leq C \Bigg(
		\int_{I_n} \int_{\Omega(t)} | \nabla (\varphi v_h) |^2 \, \dx \, \dt \\
		+ h^2 \sum_{T\in\mathcal{T}_h^{n,\Gamma}}\int_{I_n} \int_{T}
		\left| \frac{\partial \varphi}{\partial t} v_h - \Delta (\varphi v_h) \right|^2 \, \dx \, \dt
		+ h \sum_{F\in\mathcal{F}_h^{n, \Gamma}} \int_{I_n}  \int_{F}	 [\![\partial_\nu (\varphi v_h)]\!] ^2 \, \dsigma \, \dt
		\Bigg).
	\end{multline*}
\end{lemma}

\begin{proof}
	Take any $n\in\{1,\ldots,N\}$, and recall that $\mathcal{T}_h^{n,\Gamma}$ is covered by the patches $\{\Pi_i\}$ as in Assumption~\ref{asm2}, cf. Fig.~\ref{FigPatch}. We claim that the following bound holds on any patch $\Pi_i$: for any $\beta>0$, there exists $C(\beta)>0$, 	depending only on $\beta$, on $M_\Pi$ and on the regularity of the mesh 	$\mathcal{T}_h^{\mathcal{O}}$, such that for every	$w\in H^1(\Pi_i)$ with $w|_T\in H^2(T)$ for all $T\in\Pi_i$,
	\begin{equation}\label{Dir:beta}
		|w|_{1,\Pi_i^{\Gamma}}^2 \leqslant C(\beta)\left(
		|w|_{1,T_{i}}^2
		+ h\,\sum_{F\in\mathcal{F}_i}\big\| [\![ \partial_\nu w ]\!] \big\|_{F}^2
		+ h^2\,\sum_{T\in\Pi_i^{\Gamma}}\|\Delta w\|_{F}^2 \right)
		+ \beta h^2 \,\sum_{T\in\Pi_i}|w|_{2,T}^2 
	\end{equation}
with $\mathcal{F}_i$ standing for the set of mesh faces interior to $\Pi_i$. Note that this bound is invariant under the change of coordinates $x\mapsto hx+c$ (with any constant vector $c$) and under the multiplication of $w$ by any non-negative real number, as well as adding any scalar constant to it. Now, let us fix any $\beta>0$ and suppose, by contradiction, that no such $C(\beta)$ exists.
Then, for every $j\in\mathbb{N}^*$, there are a patch $\Pi_j$ (that verifies the properties of Assumption \ref{asm2},  is of size 1 ($h=1$), and centered at the origin) and a function $w_j$ normalized by $|w_j|_{1,\Pi_j^\Gamma}=1$ and $\int_{\Pi_j}w_j=0$, such that
\begin{equation}\label{Dir:contra}
	j\left(|w_j|_{1,T_{j}}^2
	+ \sum_{F\in\mathcal{F}_j}\big\| [\![ \partial_\nu w_j ]\!]\big\|_{F}^2
	+ \sum_{T\in\Pi_j^{\Gamma}}\|\Delta w_j\|_{F}^2 \right)
	+ \beta\,\sum_{T\in\Pi_j}|w_j|_{2,T}^2 \ <\ 1 .
\end{equation}
This gives at once: $\sum_{T\in\Pi_j}|w_j|_{2,T}^2\leqslant \beta^{-1}$ for all $j$, and $|w_j|_{1,T_{j}}^2
+ \sum_{F\in\mathcal{F}_j}\big\| [\![ \partial_\nu w_j ]\!]\big\|_{F}^2
+ \sum_{T\in\Pi_j^{\Gamma}}\|\Delta w_j\|_{F}^2 \longrightarrow 0$ as $j\to\infty$, meaning that $(w_j)$ is bounded in $H^2(T)$ on every cell $T\subset\Pi_j$. Extracting a subsequence, we may assume
that the patches $(\Pi_j$) converge to a limit patch $\Pi=\Pi^\Gamma\cup T_\Pi$ (of size 1 and centered at the origin) with interior facets $\mathcal{F}_\Pi$, and that $w_j\to w$ strongly in $H^1(\Pi)$ (by the Rellich theorem applied on every cell) and $w_j\rightharpoonup w$ weakly in $H^2(T)$ for every  $T\subset\Pi$.
Passing to the limit in \eqref{Dir:contra}:
\begin{itemize}
	\item $|w|_{1,\Pi^{\Gamma}}=1$, and $|w|_{1,T_{\Pi}}=0$, i.e. $w=\textrm{const}$ on $T_{\Pi}$, by the strong $H^1$ convergence;
	\item $\Delta w=0$ on each cell of $\Pi$, by the weak $L^2$ convergence	of $\Delta w_j$;
	\item $[\![\partial_\nu w]\!]=0$ on $\mathcal{F}_{\Pi}$ since the map	$v\mapsto\partial_\nu v|_F$ being compact from $H^2(T)$ into $L^2(F)$, is weakly continuous.
\end{itemize}
This implies that $w$ is harmonic in the interior of $\Pi$. Indeed, $w\in H^1(\Pi)$ is cell-wise $H^2$, with vanishing cell-wise Laplacian and vanishing normal jumps, thus $w$ satisfies the mean value property on every ball in $\Pi$ and, being continuous, it is therefore harmonic, hence real-analytic in the interior of $\Pi$ (see e.g. Chapter~1 of \cite{Axler2001}). Since 
$w$ is constant on $T_{\Pi}$ and the interior of $\Pi$ is connected, the identity theorem for real-analytic functions (same reference) gives that $w$ is constant on $\Pi$,  contradicting
$|w|_{1,\Pi^\Gamma}=1$. This proves (\ref{Dir:beta}).

Now, take any $t\in I_n$ and $v_h\in V_h^n$, write
$\varphi$ for $\varphi(t,\cdot)$, apply (\ref{Dir:beta}) to $w=\varphi v_h$ on every patch, and sum over patches. This gives
\begin{multline}\label{Dir:sum}
|\varphi v_h|_{1,\Omega_h^{n,\Gamma}}^2 \leqslant C(\beta)\left(
|\varphi v_h|_{1,\Omega(t)}^2
+ h\sum_{F\in\mathcal{F}_h^{n,\Gamma}}
\big\| [\![ \partial_\nu (\varphi v_h) ]\!]\big\|_{F}^2
+ h^2\sum_{T\in\mathcal{T}_h^{n,\Gamma}}\|\Delta (\varphi v_h)\|_{F}^2 \right)\\
+ \beta h^2 \sum_{T\in\mathcal{T}_h^{n,\Gamma,\textrm{ext}}} |\varphi v_h|_{2,T}^2 ,
\end{multline}
since the disjoint sets $\Pi_i^\Gamma$ cover $\mathcal{T}_h^{n,\Gamma}$ and $T_i$ are distinct cells of
$\mathcal{T}_h^n\setminus\mathcal{T}_h^{n,\Gamma}$, hence contained in $\Omega(t)$. We have denoted here $\mathcal{T}_h^{n,\Gamma,\textrm{ext}}=\mathcal{T}_h^{n,\Gamma} \cup T_1 \cup \cdots \cup T_{N_\Pi}$.

	Under Assumption~\ref{asm1}, there exists $C_1>0$ such that
	\begin{equation}\label{Dir:brokenH2}
		h^2\sum_{T\in\mathcal{T}_h^{n,\Gamma,\textrm{ext}}}|\varphi v_h|_{2,T}^2 \ \leqslant\ C_1\|v_h\|_{\Omega^n_h}^2 \,.
	\end{equation}
	Indeed, the Leibniz rule gives on any cell $T\in\mathcal{T}_h^{n,\Gamma,\textrm{ext}}$
	\[
	|\varphi v_h|_{2,T} \lesssim \|\varphi\|_{L^\infty(T)} |v_h|_{2,T} + |\varphi|_{W^{1,\infty}(T)} |v_h|_{1,T}
	+  |\varphi|_{W^{2,\infty}(T)} \|v_h\|_{F}
	\lesssim h |v_h|_{2,T} + |v_h|_{1,T} + \|v_h\|_{F}
	\]
	since the first and second derivatives of $\varphi$ are uniformly bounded by Assumption~\ref{asm1} and $T$ is at a distance $\sim h$ of $\partial\Omega(t)$. It remains to apply the inverse inequalities to the polynomial $v_h$ on $T$, to multiply by $h$, to square and to sum over the cells $T$, in order to prove (\ref{Dir:brokenH2}).
This allows us  to dispose of the last term of \eqref{Dir:sum}. By
\eqref{Dir:brokenH2}, and then by the Hardy inequality \eqref{Hardy} applied to $w$, which vanishes on $\Gamma(t)$,
\begin{equation}\label{Dir:betaabs}
	h^2\sum_{T\in\mathcal{T}_h^{n,\Gamma,\textrm{ext}}}|\varphi v_h|_{2,T}^2 \leqslant C_1\|v_h\|_{\Omega_h^n}^2
	= C_1\left\|\frac{w}{\varphi}\right\|_{\Omega_h^n}^2
	\leqslant C_1C_H^2\,|w|_{1,\Omega_h^n}^2 .
\end{equation}
The constants $C_1$ and $C_H$ being independent of $\beta$, we may now fix $\beta$ so that $\beta C_1C_H^2\leqslant\frac14$; the last
term of \eqref{Dir:sum} is then absorbed by its left hand side. 

It remains to make $\frac{\partial\varphi}{\partial t}v_h-\Delta w$ appear in the right hand side of (\ref{Dir:sum}). Writing $\Delta w=\frac{\partial\varphi}{\partial t}
v_h-\big(\frac{\partial\varphi}{\partial t}v_h-\Delta w\big)$ and recalling
that $\big|\frac{\partial\varphi}{\partial t}\big|\leqslant C_\varphi$ by
Assumption~\ref{asm1},
\[
h^2\sum_{T\in\mathcal{T}_h^{n,\Gamma}}\|\Delta (\varphi v_h)\|_{F}^2
\leqslant 2h^2\sum_{T\in\mathcal{T}_h^{n,\Gamma}}
\left\|\frac{\partial\varphi}{\partial t}v_h-\Delta (\varphi v_h)\right\|_{F}^2
+ 2C_\varphi^2h^2\|v_h\|_{\Omega_h^{n,\Gamma}}^2 .
\]
The last term above can be further bound $2C_\varphi^2C_H^2h^2|\varphi v_h|_{1,\Omega_h^n}^2$, using  the Hardy inequality as in \eqref{Dir:betaabs}. Its contribution to the
right hand side of \eqref{Dir:sum} is therefore absorbed by the left hand
side as soon as $2C(\beta)C_\varphi^2C_H^2h^2\leqslant\frac14$, i.e. for all
$h\leqslant h_0$, with $h_0$ depending on the already fixed $\beta$ through $C(\beta)$. We are left with
\[
|\varphi v_h|_{1,\Omega_h^{n}}^2 \leqslant C\left(
|\varphi v_h|_{1,\Omega(t)}^2
+ h\sum_{F\in\mathcal{F}_h^{n,\Gamma}}
\big\| [\![ \partial_\nu (\varphi v_h) ]\!]\big\|_{F}^2
+ h^2\sum_{T\in\mathcal{T}_h^{n,\Gamma}}
\left\|\frac{\partial\varphi}{\partial t}v_h
-\Delta(\varphi v_h)\right\|_{T}^2 \right) ,
\]
with $C=2C(\beta)$. Integrating this bound over $t\in I_n$ concludes the proof.
\end{proof}

With the help of the Lemmas above, (\ref{BIGbound}) leads now to the ``parabolic coercivity'' estimate assuming $\frac{h^2}{\tau}+\tau$ sufficiently small.

\begin{proposition}\label{coerc}
Assuming $\gamma$ big enough, $h\leqslant h_0$ and $\frac{h^2}{\tau}+\tau$ small  enough, it holds for any $U_h\in\mathbb{V}_h$
  \begin{multline*}
       a_h (U_h, U_h) \geqslant \frac{1}{2}  \int_{\Omega (t_N)} \varphi^2 (t_N) |
     \tilde{u}_h^N |^2 \dx + \frac{1}{2}  \int_{\Omega_h^0} \varphi^2 (t_0)
     | \tilde{u}_h^1 |^2 \dx \\
   + c \sum_{n = 1}^N \int_{I_n} \int_{\Omega^n_h} | \nabla u^n_h |^2 \dx \dt + \frac{1}{2} \sum_{n = 2}^N  \int_{\Omega (t_{n - 1})} \varphi^2
     (t_{n - 1})  | \tilde{u}_h^n - \tilde{u}_h^{n - 1} |^2 \dx \\
   + \frac{1}{2} \sum_{n = 1}^N \left[ \ \gamma h
   \sum_{F\in\mathcal{F}_h^{n, \Gamma}}  \int_{I_n}\int_{F}[\![\partial_\nu u^n_h]\!]^2 \dsigma \dt + \gamma h^2 \sum_{T\in\mathcal{T}_h^{n,\Gamma}} \int_{I_n} \int_{T}
     \left| \frac{\partial \varphi}{\partial t}  \tilde{u}^n_h - \Delta u^n_h
     \right|^2 \dx \dt \right]
  \end{multline*}
  with a constant $c > 0$.
\end{proposition}

\begin{proof}
  Recall from (\ref{3termsIOS}) that $a_h (U_h, U_h)$ is the sum of three
  contributions ``in'', ``out'', and ``stab'', and that the sum of the contributions ``in''  and ``out'' is bounded from below in (\ref{BIGbound}).
  Let us make precise the order in which the preceding lemmas are applied. 
  Set, for brevity,
  \[ G_n^{\mathrm{in}} = \int_{I_n} \int_{\Omega (t)} | \nabla u^n_h |^2 \dx \dt , \qquad
     G_n = \int_{I_n} \int_{\Omega^n_h} | \nabla u^n_h |^2 \dx \dt , \]
  \[ S_n = h \sum_{F\in\mathcal{F}_h^{n, \Gamma}}\int_{I_n} \int_{F} [\![\partial_\nu u^n_h]\!]^2 \dsigma \dt
     + h^2 \sum_{T\in\mathcal{T}_h^{n,\Gamma}}\int_{I_n} \int_{T} \left| \frac{\partial \varphi}{\partial t}  \tilde{u}_h^n - \Delta u^n_h \right|^2 \dx \dt , \]
  so that $a_h^{\mathrm{stab}}(U_h,U_h)=\gamma\sum_n S_n$ and the two ``parts of stabilization'' underbraced in (\ref{BIGbound}) sum up to $\frac{1}{2\varepsilon}S_n$.
  Lemma \ref{LemInverse} bounds the two underbraced terms of (\ref{BIGbound}) carrying the factor $\varepsilon$ by $c_0\varepsilon\, G_n$. Lemma \ref{LemInverseT}, applied at rank $n-1$ with $s=t_{n-1}$ and $v_h=\tilde u_h^{n-1}$, bounds the last underbraced term by $c_2\left(\frac{h^2}{\tau}+\tau\right) G_{n-1}$; here we have used the inclusion
  \[ \Omega^n_h \cap B^{n - 1} (t_{n - 1}) = \Omega_h^n\cap\big(\Omega_h^{n-1}\setminus\Omega(t_{n-1})\big)
  \subset \Omega_h^{n-1}\setminus\Omega(t_{n-1})
  \subset \Omega_h^{n-1,\Gamma} \,.\]
Lemma \ref{LemMag1} then converts these bounds into
  \[ c_0\varepsilon\, G_n + c_2\left(\frac{h^2}{\tau}+\tau\right) G_{n-1} \leqslant c_1\varepsilon\left(G_n^{\mathrm{in}} + S_n\right) + c_1\left(\frac{h^2}{\tau}+\tau\right)\left(G_{n-1}^{\mathrm{in}} + S_{n-1}\right) . \]
  Reporting this into (\ref{BIGbound}), summing in $n$ and adding the stabilization $a_h^{\mathrm{stab}}(U_h,U_h)$, we arrive at
  \begin{multline*} a_h (U_h, U_h) \geqslant \frac{1}{2}  \int_{\Omega (t_N)} \varphi^2 (t_N) |
	\tilde{u}_h^N |^2 \dx + \frac{1}{2}  \int_{\Omega_h^0} \varphi^2 (t_0) |
	\tilde{u}_h^1 |^2 \dx \\
	+ \sum_{n = 1}^N \left( 1-c_1 \varepsilon - c_1 \left(\frac{h^2}{\tau}+\tau\right) \right) G_n^{\mathrm{in}} +
	\sum_{n = 2}^N \frac{1}{2}  \int_{\Omega (t_{n - 1})} \varphi^2 (t_{n - 1})  |
	\tilde{u}_h^n - \tilde{u}_h^{n - 1} |^2 \dx \\
	+ \sum_{n = 1}^N \left( \gamma -c_1\varepsilon - c_1\left(\frac{h^2}{\tau}+\tau\right) - \frac{1}{2 \varepsilon} \right) S_n .
\end{multline*}
    We now choose $\varepsilon$ small enough and assume $\frac{h^2}{\tau}+\tau$ small enough so that $c_1\varepsilon+c_1\left(\frac{h^2}{\tau}+\tau\right)\leqslant \frac12$, and then $\gamma$ big enough so that $\gamma-c_1\varepsilon-c_1\left(\frac{h^2}{\tau}+\tau\right)-\frac{1}{2\varepsilon}\geqslant \frac{\gamma}{2}+\frac12$. A last application of Lemma \ref{LemMag1}, in the form $G_n \leqslant c_1 (G_n^{\mathrm{in}} + S_n)$, allows us to replace $\frac12 (G_n^{\mathrm{in}} + S_n)$ by $\frac{1}{2c_1} G_n$, which gives the announced coercivity estimate with $c=\frac{1}{2c_1}$, the gradient term being now integrated over the whole of $\Omega_h^n$ and the remaining $\frac{\gamma}{2}S_n$ giving the last line of the statement.
\end{proof}

\subsection{Consistency and convergence estimates}
\label{Sect:Consistency}

The goal is to combine the coercivity proven in the previous section with
interpolation estimates, and to obtain optimal convergence in the $L^2 (H^1)$ norm.

We suppose that the exact solution $u$ is in the space $\mathcal{H}^k(Q_T)$ as defined in (\ref{Hkspace}) and
extend it from $Q_n$ to $I_n \times \Omega_h^n$ for every $n$. Let us denote this extension by $u^n$. {We can safely assume that all these extensions are compatible between the time slabs}
\begin{equation}\label{compatExt}
	{u^n(t_{n-1},\cdot)=u^{n-1}(t_{n-1},\cdot) \text{ on } \Omega_h^n\cap\Omega_h^{n-1}, \quad n\geqslant 2, \qquad u^1(t_0,\cdot)=u^0 \text{ on } \Omega_h^0 ,}
\end{equation}
and satisfy, in accordance with (\ref{Hkspace}), 
\begin{equation}
	\label{regu} u^n \in H^2 (I_n ; H^{k - 1} (\Omega_h^n)) \cap H^1 (I_n ; H^k
	(\Omega_h^n)) \cap L^2 (I_n ; H^{k + 1} (\Omega_h^n)) .
\end{equation}
Each $u^n$ satisfies $\frac{\partial u^n}{\partial t} - \Delta u^n =
\tilde{f}^n$ on $I_n \times \Omega_h^n$ with some $\tilde{f}^n \in L^2 (I_n ; H^{k - 1} (\Omega_h^n))$.  The extension $\tilde{f}^n$ coincides with $f$ inside $\Omega (t)$ for all $t \in I_n$.
 We also introduce $\tilde{u}^n = u^n / \varphi$ and note that
\begin{equation}
	\label{regutil} \tilde{u}^n \in H^1 (I_n ; H^{k - 1} (\Omega_h^n)) \cap L^2
	(I_n ; H^k (\Omega_h^n)) .
\end{equation}
To prove this regularity of {$\tilde{u}^n$}, we start by applying {a Hardy-type inequality, similar to that of Lemma \ref{LemHardy}, but generalized to higher orders in space as in  \cite{duprez2020phi}},  at any fixed time $t \in I_n$: if ${u^n} (t) \in H^{k + 1}
(\Omega^n_h)$, then ${\tilde{u}^n} (t) \in H^k (\Omega^n_h)$. In view of (\ref{regu}), this implies $\tilde{u}^n \in L^2 (I_n ; H^k (\Omega_h^n))$.
Now, looking at the identity
\[ \frac{\partial {u^n}}{\partial t} - \frac{\partial \varphi}{\partial t} {\tilde{u}^n}
= \varphi \frac{\partial {\tilde{u}^n}}{\partial t} \]
we see that its LHS is in $L^2 (I_n ; H^k (\Omega_h^n))$. Thus, applying the Hardy inequality again, we conclude that $\frac{\partial {\tilde{u}^n}}{\partial	t}$ is in $L^2 (I_n ; H^{k - 1} (\Omega_h^n))$. Note that the norms of ${u^n}$ and ${\tilde{ u}^n}$ in \eqref{regu} and \eqref{regutil} respectively can be assumed to depend continuously on the norm of $u$ in $\mathcal{H}^k(Q_T)$.

We put the slices $u^n = \varphi \tilde{u}^n$ in {$U=(u^n = \varphi \tilde{u}^n)_{n = 1, \ldots, N}$.}
The form $a_h$ defined by \eqref{ahalt}, can be also written as follows, for all $U, V$ with $U$ as above and $V=(v^n = \varphi \tilde{v}^n)_{n	= 1, \ldots, N}$:
\begin{multline}\label{ahgen}
	a_h (U, V) = \sum_{n = 1}^N \left[ \int_{I_n} \int_{\Omega_h^n} \frac{\partial u^n}{\partial t} v^n \dx \dt \right. +
\int_{\Omega_h^n \cap \Omega_h^{n - 1}} (u^n (t_{n - 1}) -\mathbf{1}_{n >1} u^{n - 1} (t_{n - 1})) v^n (t_{n - 1}) \dx \\
 + \int_{I_n} \int_{\Omega_h^n} 
 {\nabla u^n\cdot\nabla v^n}\dx \dt - \int_{I_n} \int_{\partial \Omega_h^n} \partial_{\nu}{u^n}v^n \dsigma \dt \\
\left. + \gamma h\sum_{F\in\mathcal{F}_h^{n, \Gamma}} \int_{I_n}  \int_{F} [\![\partial_\nu u^n]\!] [\![\partial_\nu v^n]\!] \dsigma \dt + \gamma h^2  \sum_{T\in\mathcal{T}_h^{n,\Gamma}} \int_{I_n} \int_{T} \left( \frac{\partial \varphi}{\partial t}
\tilde{u}^n - \Delta u^n \right)  \left( \frac{\partial \varphi}{\partial t}  \tilde{v}^n - \Delta v^n \right) \dx \dt\right].
\end{multline}
In particular, this expression coincides with \eqref{ahalt} for {$U_h,V_h\in\mathbb{V}_h$, since $\int_{I_n}\varphi\frac{\partial \varphi}{\partial t}\dt=\frac{\varphi^2(t_n)-\varphi^2(t_{n-1})}{2}$ and $\tilde u^n_h$ does not depend on $t$}. Applying it to the exact solution $U$ and any test function {$V_h\in\mathbb{V}_h$, integrating by parts in $x$ and using \eqref{compatExt} together with $[\![\partial_\nu u^n]\!]=0$}, we observe
\begin{multline*}
	a_h (U, V_h) = \sum_{n = 1}^N \left[ \int_{I_n} \int_{\Omega_h^n} \tilde{f}^n v_h^n \dx \dt + \gamma h^2  \sum_{T\in\mathcal{T}_h^{n,\Gamma}}\int_{I_n} \int_{T} \left( \tilde{f}^n - \varphi \frac{\partial \tilde{u}^n}{\partial t}
\right)  \left( \frac{\partial \varphi}{\partial t}  \tilde{v}_h^n - \Delta
v_h^n \right) \dx \dt \right] \\
 + \int_{\Omega_h^0} u^0 \varphi (t_0)  \tilde{v}_h^1 \dx .
\end{multline*}
Let $\hat{u}^n_h = \frac{1}{\tau}  \int_{t_{n - 1}}^{t_n} I_h  \tilde{u}^n (t)
\dt$ with {Clément} interpolant $I_h$ into the {$\mathbb{P}_k$} finite elements on mesh $\mathcal{T}_h^n$. Set $\check{u}^n_h = \varphi \hat{u}^n_h$ and put all $\check{u}^n_h$ into {$\check{U}_h\in\mathbb{V}_h$.} We have then
\begin{multline}\label{GalOrt}
	a_h (\check{U}_h   - U_h, V_h) = a_h (\check{U}_h - U, V_h) \\
 + \sum_{n = 1}^N \left[ \int_{I_n} \int_{\Omega_h^n} (\tilde{f}^n - f) \varphi \tilde{v}_h^n \dx \dt + \gamma h^2 \sum_{T\in\mathcal{T}_h^{n,\Gamma}} \int_{I_n}
\int_{T} \left( \tilde{f}^n - \varphi \frac{\partial
	\tilde{u}^n}{\partial t} - f \right)  \left( \frac{\partial \varphi}{\partial
	t}  \tilde{v}_h^n - \Delta v_h^n \right) \dx \dt \right] .
\end{multline}

Now let us put $V_h = \check{U}_h - U_h$, apply the coercivity Proposition \ref{coerc} on the LHS of (\ref{GalOrt}), and perform some rearrangements on its RHS, in order to be able to apply the interpolation estimates to $\check{U}_h   - U_h$. A special care
should be taken when dealing with the terms in $a_h (\check{U}_h - U, V_h)$ corresponding to the derivative in time, i.e. in the first line in the expression for $a_h (\check{U}_h - U, V_h)$, as in (\ref{ahgen}). Let us look more closely at the terms in this line, taking any $n > 1$ {and recalling \eqref{compatExt}, which makes the contributions of $u^n$ and $u^{n-1}$ at $t_{n-1}$ cancel}:
{\begin{align*} &\int_{I_n} \int_{\Omega_h^n}  \frac{\partial (\check{u}^n_h -
	u^n)}{\partial t} v_h^n \dx \dt + \int_{\Omega_h^n \cap
	\Omega_h^{n - 1}} (\check{u}^n_h - \check{u}^{n - 1}_h)  (t_{n - 1}) v^n_h
(t_{n - 1}) \dx \\
 &= \int_{I_n} \int_{\Omega_h^n}  \frac{\partial \varphi}{\partial t}
(\hat{u}_h^n - \tilde{u}^n) \varphi \tilde{v}_h^n \dx \dt -
\int_{I_n} \int_{\Omega_h^n \setminus \Omega_h^{n - 1}} \varphi \frac{\partial
	\tilde{u}^n}{\partial t} \varphi \tilde{v}^n_h \dx \dt \\
 & \qquad + \int_{\Omega_h^n \cap \Omega_h^{n - 1}} \varphi(t_{n-1}) (\hat{u}_h^n  - \hat{u}_h^{n - 1}) \varphi(t_{n-1}) \tilde{v}_h^n \dx  -
\int_{I_n} \int_{\Omega_h^n \cap \Omega_h^{n - 1}} \varphi \frac{\partial
	\tilde{u}^n}{\partial t} \varphi \tilde{v}^n_h \dx \dt \\
 &= \int_{I_n} \int_{\Omega_h^n}  \frac{\partial \varphi}{\partial t}
(\hat{u}_h^n - \tilde{u}^n) \varphi \tilde{v}_h^n \dx \dt -
\int_{I_n} \int_{\Omega_h^n \setminus \Omega_h^{n - 1}} \varphi \frac{\partial
	\tilde{u}^n}{\partial t} \varphi \tilde{v}^n_h \dx \dt \\
 & \qquad + \int_{I_n} \int_{\Omega_h^n \cap \Omega_h^{n - 1}} \varphi \left(
\frac{{\hat{u}_h^n}  - \hat{u}_h^{n - 1}}{\tau} - \frac{\partial
	\tilde{u}^n}{\partial t} \right) \varphi  \tilde{v}^n_h \dx
\dt \\
 & \qquad + \int_{\Omega_h^n \cap \Omega_h^{n - 1}} \left( \varphi^2 (t_{n - 1})
- \frac{1}{\tau} \int_{I_n}^{} \varphi^2 (t) \dt \right) \left(
{\hat{u}_h^n}  - \hat{u}_h^{n - 1} \right) \tilde{v}^n_h
\dx \\
 &\lesssim \left\| \frac{\partial \varphi}{\partial t} (\tilde{u}^n -
\hat{u}_h^n) \right\|_{I_n \times \Omega_h^n} \| \nabla v_h^n
\|_{I_n \times \Omega_h^n} + h  \left\| \varphi \frac{\partial
	\tilde{u}^n}{\partial t} \right\|_{I_n \times (\Omega_h^n \setminus
	\Omega_h^{n - 1})} \| \nabla v_h^n \|_{I_n \times \Omega_h^n} \\
 & \qquad + \left\| \varphi \left( \frac{\partial \tilde{u}^n}{\partial t} -
\frac{{\hat{u}_h^n}  - \hat{u}_h^{n - 1}}{\tau} \right)  \right\|
_{I_n \times (\Omega_h^n \cap \Omega_h^{n - 1})} \| \nabla v_h^n
\|_{I_n \times \Omega_h^n} \\
 & \qquad + \tau^{3 / 2} \left\| \frac{{\hat{u}_h^n}
	- \hat{u}_h^{n - 1}}{\tau} \right\|_{\Omega_h^n \cap \Omega_h^{n -
		1}} \| \nabla v_h^n \|_{I_n \times \Omega_h^n} .
\end{align*}}
{We have used here the following bounds for $v^n_h = \varphi \tilde{v}^n_h$, 
\[ \|v^n_h \|_{I_n \times \Omega_h^n} \lesssim \| \nabla v_h^n \|_{I_n \times
	\Omega_h^n}, \qquad \|v^n_h \|_{I_n \times (\Omega_h^n \setminus
	\Omega_h^{n - 1})} \lesssim h \| \nabla v_h^n \|_{I_n \times \Omega_h^n} \,,\]
which are the inequalities of  Corollary \ref{LemPoin}, integrated in time.  	
We have also observed that
\[ \left| \varphi^2 (t_{n - 1}) - \frac{1}{\tau}  \int_{I_n}
\varphi^2 (t) \dt \right| = \frac{1}{\tau} \left|
\int_{I_n} \int_{t_{n - 1}}^{t} \frac{\partial
	\varphi^2}{\partial t} (s) \ds \dt \right| \lesssim \sqrt{\tau}
\left( \int_{I_n} \varphi^2 (t) \dt \right)^{1 / 2} \]
by the uniform bound of $\frac{\partial \varphi }{\partial t}$ from  Assumption \ref{asm1}, entailing
\begin{multline*}
 \left| \int_{\Omega_h^n \cap \Omega_h^{n - 1}} \left( \varphi^2 (t_{n - 1})
- \frac{1}{\tau}  \int_{I_n} \varphi^2 (t) \dt \right) 
(\hat{u}_h^n - \hat{u}_h^{n - 1})  \tilde{v}^n_h \dx \right| 
\\
 \lesssim \sqrt{\tau} \| \hat{u}_h^n - \hat{u}_h^{n - 1} \|_{\Omega_h^n \cap
	\Omega_h^{n - 1}} \left( \int_{I_n} \int_{\Omega_h^n \cap
	\Omega_h^{n - 1}} \varphi^2 | \tilde{v}^n_h |^2 \dx \dt
\right)^{1 / 2} \lesssim \sqrt{\tau} \| \hat{u}_h^n - \hat{u}_h^{n - 1}
\|_{\Omega_h^n \cap \Omega_h^{n - 1}} \| \nabla v_h^n \|_{I_n\times\Omega_h^n} 
\end{multline*}
thanks again to Corollary \ref{LemPoin}.}
The estimates in the case $n = 1$ are the same, if we compensate for the
absence of $\check{u}^0_h - u^0 (t_0)$ by adding
\[ \int_{\Omega_h^0} (\check{u}^0_h - u^0) v^1 (t_0) \dx \leqslant \|
u^0 - \varphi  (t_0) \hat{u}_h^0 \|_{\Omega_h^0} \| \varphi  (t_0) \tilde{v}_h^1
\|_{\Omega_h^0} \]
with $\tilde{u}^0 = u^0 / \varphi  (t_0)$.

Using these estimates on the first line of $a_h(\check{U}_h-U,V_h)$, i.e. on the first line of the RHS of (\ref{GalOrt}) with
$V_h = \check{U}_h - U_h$, proceeding in the same way in the other lines of the
RHS, which is more straightforward, and applying Young inequalities to all the
resulting terms, which gives the norms of the test functions $v_h^n$, which
can be absorbed by the same arising on the LHS thanks to Proposition \ref{coerc},
we arrive at
\begin{multline}\label{BeforeInterp}
	 \| \varphi (t_N) \tilde{v}_h^N \|_{\Omega (t_N)}^2 + \|
\varphi (t_0) \tilde{v}_h^1 \|_{\Omega^0_h}^2 + \sum_{n = 1}^N \| \nabla v^n_h
\|_{I_n \times \Omega_h^n}^2 + \sum_{n = 2}^N  \| \varphi (t_{n - 1})
(\tilde{v}_h^n - \tilde{v}_h^{n - 1}) \|_{\Omega (t_{n - 1})}^2
\\
+ \sum_{n = 1}^N \left[ \hspace{0.17em} h \sum_{F\in\mathcal{F}_h^{n, \Gamma}}  \| [\![\partial_\nu v^n_h]\!] \|_{I_n \times F}^2 + h^2  \left\|
\frac{\partial \varphi}{\partial t}  \tilde{v}^n_h - \Delta v^n_h \right\|_{I_n \times {\Omega_h^{n,\Gamma}}}^2 \right] \\
 \lesssim \sum_{n = 1}^N \Bigg[
   \left\| \frac{\partial \varphi}{\partial t} (\tilde{u}^n - \hat{u}_h^n) \right\|_{I_n \times \Omega_h^n}^2 +
\left\| \varphi \left( \frac{\partial \tilde{u}^n}{\partial t} -
\frac{{\hat{u}_h^n}  - \hat{u}_h^{n - 1}}{\tau} \right)  \right\|_{I_n \times (\Omega_h^n \cap \Omega_h^{n - 1})}^2
+ h^2 \left\| \varphi \frac{\partial \tilde{u}^n}{\partial t} \right\|_{I_n	\times {\Omega_h^{n,\Gamma}}}^2
\\
+ \tau^3 \left\| \frac{{\hat{u}_h^n}  - \hat{u}_h^{n - 1}}{\tau} \right\|_{\Omega_h^n\cap \Omega_h^{n - 1}}^2
+ \| \nabla (u^n - \check{u}^n_h)\|_{I_n \times \Omega_h^n}^2
+ h \|\partial_\nu ({u^n} - \check{u}^n_h)\|_{I_n \times \partial \Omega_h^n}^2 \\
+ h\sum_{F\in\mathcal{F}_h^{n, \Gamma}}  \| [\![\partial_\nu (u^n - \check{u}^n_h) ]\!] \|_{I_n \times F}^2
+ h^2  \left\| \frac{\partial\varphi}{\partial t} (\tilde{u}^n - \hat{u}^n_h) \right\|^2_{I_n \times{\Omega_h^{n,\Gamma}}}
+ h^2 \| \Delta (u^n - \check{u}^n_h) \|^2_{I_n \times {\Omega_h^{n,\Gamma}}} \\
+ h^2 \|f - \tilde{f}^n \|^2_{I_n	\times {\Omega_h^{n,\Gamma}}} \Bigg]
+ \| u^0 - \varphi  (t_0) I_h \tilde{u}^0 \|_{\Omega_h^0}^2.
\end{multline}
Note, in particular, that the second line of (\ref{GalOrt}) has resulted in
$h^2 \|f - \tilde{f}^n \|^2_{I_n \times {\Omega_h^{n,\Gamma}}}$ \ since $f =
\tilde{f}^n$ outside ${\Omega_h^{n,\Gamma}}$, and one gets a factor $O (h)$
from the Poincar{\'e} inequality for $\varphi \tilde{v}^n_h$ on
$\Omega_h^{n,\Gamma}$ {(first inequality of Corollary \ref{LemPoin})}. The contribution of $\varphi \frac{\partial
	\tilde{u}^n}{\partial t}$ on ${\Omega_h^{n,\Gamma}}$ has been combined
with the same on $\Omega_h^n \setminus \Omega_h^{n - 1}$ knowing that
$\Omega_h^n \setminus \Omega_h^{n - 1} \subset {\Omega_h^{n,\Gamma}}$.

Now, let us list the interpolation properties we need to bound the RHS of (\ref{BeforeInterp}):
\begin{lemma}	\label{LemInterp}
Recall $\hat{u}^n_h = \frac{1}{\tau}  \int_{I_n} I_h
	\tilde{u}^n (t) \dt$. We have
	\begin{equation}
		\label{L2intutil} \| \tilde{u}^n - \hat{u}^n_h \|_{I_n \times
			\Omega_h^n} \lesssim \tau \| \tilde{u}^n \|_{H^1 (I_n ; L^2 (\Omega^n_h))}
		+ h^k \| \tilde{u}^n \|_{L^2 (I_n ; H^k (\Omega^n_h))}
	\end{equation}

	\begin{equation}
		\label{H1intu} \| \nabla (u^n - \check{u}^n_h)\|_{I_n \times \Omega_h^n}
		\lesssim \tau \|u^n \|_{H^1 (I_n ; H^1 (\Omega^n_h))} + h^k |u^n |_{L^2
			(I_n ; H^{k + 1} (\Omega^n_h))}
	\end{equation}

	\begin{equation}
		\label{H2intu} \| \Delta (u^n - \check{u}^n_h) \|_{I_n \times{\Omega_h^{n,\Gamma}}}
		\lesssim \frac{\tau}{h} \|u^n \|_{H^1 (I_n; H^1(\Omega^n_h))}
			+ h^{k - 1} |u^n |_{L^2 (I_n ; H^{k + 1} (\Omega^n_h))}
	\end{equation}

	\begin{multline}\label{Edgesintu}
		 \|\partial_{\nu} ({u^n} - \check{u}^n_h)\|_{I_n \times
			\partial \Omega_h^n}
			+ \sum_{F\in\mathcal{F}_h^{n, \Gamma}}  \|[\![\partial_{\nu} ({u^n} - \check{u}^n_h) ]\!]\|_{I_n\times F} \\
			 \lesssim \frac{\tau}{\sqrt h} \| {u^n} \|_{H^1 (I_n; H^1 (\Omega^n_h))}
			 + h^{k - 1 / 2} \| {u^n} \|_{L^2 (I_n ; H^{k + 1}(\Omega^n_h))}
	\end{multline}

	\begin{multline}\label{L2intdudt}
	 \left\| \varphi \left( \frac{\partial \tilde{u}^n}{\partial t} -
	\frac{{\hat{u}_h^n}  - \hat{u}_h^{n - 1}}{\tau} \right) \right\|_{I_n	\times (\Omega_h^n \cap \Omega_h^{n - 1})}  \lesssim \tau (\|u\|_{H^2 (I_{n -
			1} \cup I_n ; L^2 (\Omega^n_h \cap \Omega_h^{n - 1}))} +\| \tilde{u} \|_{H^1
		(I_{n - 1} \cup I_n ; L^2 (\Omega^n_h \cap \Omega_h^{n - 1}))})
	\\
		 + h^k \left( \left\| \frac{\partial u}{\partial t}
		\right\|_{L^2 (I_{n - 1} \cup I_n ; H^k (\Omega^n_h \cap \Omega_h^{n -
				1}))} + \| \tilde{u} \|_{L^2 (I_{n - 1} \cup I_n ; H^k (\Omega^n_h \cap
			\Omega_h^{n - 1}))} \right)
	\end{multline}
and
\begin{equation}
	\label{Boundedintdudt} \left\| \frac{{\hat{u}_h^n}  - \hat{u}_h^{n -
			1}}{\tau} \right\|_{\Omega_h^n \cap \Omega_h^{n - 1}} \lesssim \|
	\tilde{u} \|_{H^1 (I_n \cup I_{n - 1} ; L^2 (\Omega^n_h \cap \Omega_h^{n -
			1}))}.
\end{equation}
\end{lemma}

\begin{proof}
	To prove (\ref{L2intutil}), we observe first for any $t \in I_n$
	\[ \| \tilde{u}^n (t) - I_h  \tilde{u}^n (t) \|_{\Omega_h^n} \lesssim h^k
	\| \tilde{u}^n (t) \|_{H^k (\Omega^n_h)}. \]
	This gives, after integrating in $t$,
	\[ \left\|  \frac{1}{\tau}  \int_{t_{n - 1}}^{t_n}  \tilde{u}^n (t)
	\dt - \hat{u}^n_h \right\|_{\Omega_h^n} \lesssim
	\frac{h^k}{\sqrt{\tau}} \| \tilde{u}^n (t) \|_{L^2 (I_n ; H^k
		(\Omega^n_h))}. \]
	By triangle inequality,
	\[ \| \tilde{u}^n - \hat{u}^n_h \|_{I_n \times \Omega_h^n} \leqslant
	\left\|  \tilde{u}^n - \frac{1}{\tau}  \int_{t_{n - 1}}^{t_n}
	\tilde{u}^n (t) \dt \right\|_{I_n \times \Omega_h^n} +
	\sqrt{\tau} \left\|  \frac{1}{\tau}  \int_{t_{n - 1}}^{t_n}
	\tilde{u}^n (t) \dt - \hat{u}^n_h \right\|_{\Omega_h^n} \]
	so that (\ref{L2intutil}) follows by observing that
	\[ \left\|  \tilde{u}^n - \frac{1}{\tau}  \int_{t_{n - 1}}^{t_n}
	\tilde{u}^n (t) \dt \right\|_{I_n \times \Omega_h^n}
	\leqslant \tau \| \tilde{u}^n \|_{H^1 (I_n ; L^2 (\Omega^n_h))}. \]

	The proof of \eqref{H1intu} is similar, starting from
	\[ \| \nabla (u^n (t) - \varphi (t) I_h \tilde{u}^n (t))\|_{\Omega_h^n} \lesssim h^k |u^n (t) |_{H^{k + 1} (\Omega^n_h)}, \]
	which is valid for any $t \in I_n$. Indeed, as proven in \cite{duprez2023phi}, we have
	for any function $v \in H^{k + 1} (\Omega_h^n)$ vanishing on $\Gamma (t)$,
	\begin{equation}\label{fromStokes}
		\left\| \nabla \left( v - \varphi (t) I_h  \frac{v}{\varphi (t)}
		\right) \right\|_{\Omega_h^n} \lesssim h^k | v |_{H^{k + 1} (\Omega^n_h)}.
	\end{equation}

	Bound (\ref{H2intu}) can be deduced from (\ref{H1intu}) by an inverse inequality if $u^n$ is piecewise-polynomial and of class $C^1$ on $\Omega_h^n$ at all time $t\in I_n$. Now, any $u^n\in L^2 (I_n ; H^{k + 1})$ can be approximated by a function of this type, leading to an additional contribution of order $h^{k - 1} |u^n |_{L^2 (I_n ; H^{k + 1})}$ in (\ref{H2intu}), which proves (\ref{H2intu}) for a general $u^n$.  Applying the trace inequality to $\partial_{\nu} (u^n(t) - \varphi I_h \tilde{u}^n(t))$ at each time $t\in I_n$ on each mesh edge, and integrating on $t$, we see that (\ref{Edgesintu}) follows from (\ref{H1intu}) and (\ref{H2intu}).

	The proof of (\ref{L2intdudt}) is more tedious because of possible lack of regularity for $\frac{\partial^2 \tilde{ u}}{\partial t^2}$. We start by	observing for any $t \in I_n$,
	\begin{multline*}
		\left\| \varphi (t) \left( \frac{\partial \tilde{u}}{\partial t} (t)
	- I_h \frac{\tilde{u} (t) - \tilde{u} (t - \tau)}{\tau} \right) \right\|_{\Omega_h^n \cap \Omega_h^{n - 1}} \leqslant \left\|
	\frac{\partial u}{\partial t} (t) - \frac{u (t) - u (t - \tau)}{\tau}
	\right\|_{\Omega_h^n \cap \Omega_h^{n - 1}} \\
	 + \frac{1}{\tau} \left\| \int_{t - \tau}^t \left( \frac{\partial
		u}{\partial t} (s) - \frac{\partial \varphi}{\partial t} (s) \tilde{u} (s)
	\right) \ds - \int_{t - \tau}^t \varphi (s) I_h \frac{\frac{\partial
			u}{\partial t} (s) - \frac{\partial \varphi}{\partial t} (s) \tilde{u}
		(s)}{\varphi (s)} \ds \right\|_{\Omega_h^n \cap \Omega_h^{n -
			1}} \\
	 + \frac{1}{\tau} \left\| \int_{t - \tau}^t (\varphi (s) - \varphi (t))
	I_h \frac{\partial \tilde{u}}{\partial t} (s) \ds \right\|_{\Omega_h^n \cap \Omega_h^{n - 1}} \\
	 + \frac{1}{\tau} \left\| \int_{t - \tau}^t \frac{\partial
		\varphi}{\partial t} (s) (\tilde{u} (s) - \tilde{u} (t)) \ds \right\|_{\Omega_h^n \cap \Omega_h^{n - 1}} + \left\| \left(
	\frac{1}{\tau} \int_{t - \tau}^t \frac{\partial \varphi}{\partial t} (s)
	\ds - \frac{\partial \varphi}{\partial t} (t) \right) \tilde{u} (t)
	\right\|_{\Omega_h^n \cap \Omega_h^{n - 1}} \\
	 \lesssim \sqrt{\tau} \|u\|_{H^2 (t - \tau, t ; L^2 (\Omega^n_h \cap
		\Omega_h^{n - 1}))} + \frac{h^k}{\sqrt{\tau}} \left\| \frac{\partial
		u}{\partial t} - \frac{\partial \varphi}{\partial t} \tilde{u} \right\|_{L^2
		(t - \tau, t ; H^k (\Omega^n_h \cap \Omega_h^{n - 1}))}
	 + \sqrt{\tau} \| \tilde{u} \|_{H^1 (t - \tau, t ; L^2 (\Omega^n_h \cap
		\Omega_h^{n - 1}))} .
	\end{multline*}
	We have used here a bound, inherited from \cite{duprez2023phi} and similar to
	(\ref{fromStokes}): for any function $v \in H^k (\Omega_h^n)$ vanishing on
	$\Gamma (t)$,
	\begin{equation}\label{fromStokesL2}
	 \left\|v - \varphi (t) I_h  \frac{v}{\varphi (t)} \right\|_{\Omega_h^n} \lesssim h^k | v
	|_{H^k (\Omega^n_h)}.
	\end{equation}
	Integrating the bound above in time over $I_n$ gives (\ref{L2intdudt}).

	Finally, observing that
	\[ \frac{{\hat{u}_h^n}  - \hat{u}_h^{n - 1}}{\tau} = \frac{1}{\tau}
	\int_{I_n} I_h \frac{\tilde{u} (t) - \tilde{u} (t - \tau)}{\tau} \dt
	\]
	and combining with bounds for $I_h$, leads to (\ref{Boundedintdudt}).
\end{proof}

Using the estimates of the preceding Lemma in (\ref{BeforeInterp}) together with the regularity of $\tilde{u}$ in (\ref{regutil}) inherited from that of $u$ in (\ref{regu}), leads to
\begin{multline}\label{H1errhitp}
	\left( \sum_{n = 1}^N \| \nabla u^n_h - \nabla \check{u}^n_h \|_{I_n \times\Omega_h^n}^2 \right)^{\frac{1}{2}}
	= \left( \sum_{n = 1}^N \| \nabla v^n_h \|_{I_n \times \Omega_h^n}^2 \right)^{\frac{1}{2}} \\
	\lesssim (\tau + h^k) \Bigg( \sum_{n = 1}^N \Big[ \|u^n \|^2_{H^2 (I_n ; H^{k - 1} (\Omega^n_h))}
	+\|u^n \|^2_{H^1 (I_n ; H^k (\Omega^n_h))} +\|u^n \|^2_{L^2 (I_n ; H^{k +1} (\Omega^n_h))} \Big] \Bigg)^{\frac{1}{2}} \\
	+ \Bigg( \sum_{n = 1}^N \Big[ h^2 \left\| \varphi \frac{\partial	\tilde{u}^n}{\partial t} \right\|_{I_n \times {\Omega_h^{n,\Gamma}}}^2
	+ h^2 \|f - \tilde{f}^n \|^2_{I_n \times {\Omega_h^{n,\Gamma}}} \Big] \Bigg)^{\frac{1}{2}}
	+ h^{k} \| u^0 \|_{H^{k+1} (\Omega_h^0)} .
\end{multline}
Note that the summation over $n$ now takes place inside the square roots, and that the two contributions of the last line are \emph{not} multiplied by $(\tau+h^k)$: they come from the terms of (\ref{BeforeInterp}) which carry no interpolation factor and have to be estimated separately. This is what we do now.

The first of them is bounded observing that $\varphi$ is of order $h$ on ${\Omega_h^{n,\Gamma}}$, thanks to Assumptions \ref{asm0} and \ref{asm2}:
\[ h\left\| \varphi \frac{\partial \tilde{u}^n}{\partial t} \right\|_{I_n \times {\Omega_h^{n,\Gamma}}} \lesssim h^2  \left\|
\frac{\partial \tilde{u}^n}{\partial t} \right\|_{I_n \times
	{\Omega_h^{n,\Gamma}}}
\lesssim \tau \|{\tilde u}^n \|_{H^1 (I_n ; {L^2}(\Omega^n_h))} ,\]
since we have supposed $h^2 \lesssim \tau$. Finally, the norm of   $f - \tilde{f}^n$ can be estimated as follows, recalling that $f = \tilde{f}^n$ on $\Omega (t)$ for all $t$ and that $\mathcal{T}_h^{n,\Gamma}$ is an $O(h)$ band around $\Gamma(t)$ by Assumption \ref{asm2}:
\begin{multline*} \|f - \tilde{f}^n \|_{I_n \times {\Omega_h^{n,\Gamma}}} \lesssim h^{k -
	1} \|f - \tilde{f}^n \|_{L^2 (I_n ; H^{k - 1} (\Omega^n_h))} \\ \lesssim h^{k -
	1} {\left(\|f\|_{L^2 (I_n ; H^{k - 1} (\Omega^n_h))} + \|u^n\|_{H^1(I_n;H^{k-1}(\Omega_h^n))} + \|u^n\|_{L^2 (I_n; H^{k + 1} (\Omega^n_h))}\right)} .
\end{multline*}
{Here, the first inequality uses that $g:=f-\tilde f^n$ vanishes on $\Omega(t)$, so that its derivatives up to the order $k-2$ vanish on $\Gamma(t)$ and a Taylor expansion in the $O(h)$ band gives $\|g\|_{\Omega_h^{n,\Gamma}}\lesssim h^{k-1}\|g\|_{H^{k-1}(\Omega_h^n)}$; the second one uses $\tilde f^n=\frac{\partial u^n}{\partial t}-\Delta u^n$.}
{Both contributions of the last line of (\ref{H1errhitp}) are therefore of order $\tau + h^{k}$, so that}
\[ {\left( \sum_{n = 1}^N \| \nabla u^n_h - \nabla \check{u}^n_h \|_{I_n \times\Omega_h^n}^2 \right)^{\frac{1}{2}}
\lesssim  h^k \| u^0 \|_{H^{k+1} (\Omega_h^0)} + (\tau+h^k)\left( \|u\|_{\mathcal{H}^k (Q_T)} + \Big(\sum_{n = 1}^N \| f \|_{L^2 (I_n ; H^{k - 1} (\Omega^n_h))}^2\Big)^{\frac12} \right) .} \]
{Combining this with the interpolation error estimate (\ref{H1intu}) by the triangle inequality, and observing that $|u-u^n_h|_{H^1(\Omega(t))}\leqslant \|\nabla(u^n-u^n_h)\|_{\Omega_h^n}$ since $\Omega(t)\subset\Omega_h^n$ for $t\in I_n$, gives the error estimate from Theorem \ref{thm:error}.}

\section{Numerical results}
\label{Sect:Numerical_results}

In this section, we focus on the FEniCS \cite{logg2012automated} implementation of our approach, with the aim of assessing the efficiency of our scheme on three test cases\footnote{Simulations performed with FEniCS 2019.1.0, Python 3.11, on an Intel Core Ultra 7 165H CPU (16 cores, x86\_64 architecture) with 32~GB of RAM, running Linux.}: two in two dimensions (a translating circle and a rotating star) and one in three dimensions (a moving popcorn-shaped domain), see Figure~\ref{fig:geometry}. The code used to perform the numerical simulations presented in this article is available in the {GitHub} repository
\vspace{1em}
\begin{center}
\url{https://github.com/PhiFEM/publication_Heat_MovingDomain_fenics}
\end{center}
\vspace{1em}

\begin{figure}[h!]
    \centering

    \includegraphics[width=0.32\linewidth]{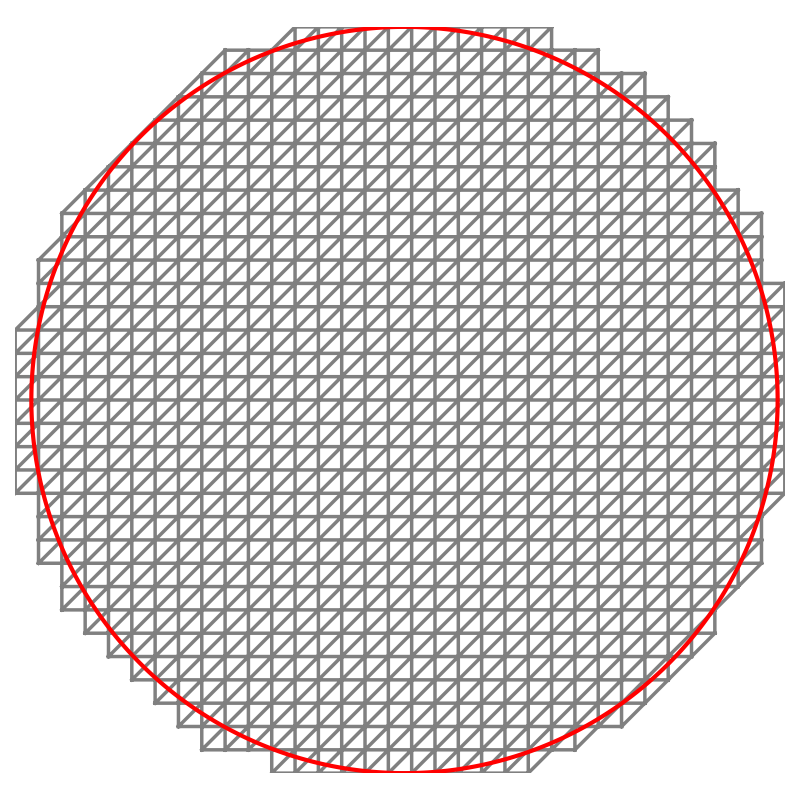}
    \hfill
    \includegraphics[width=0.32\linewidth]{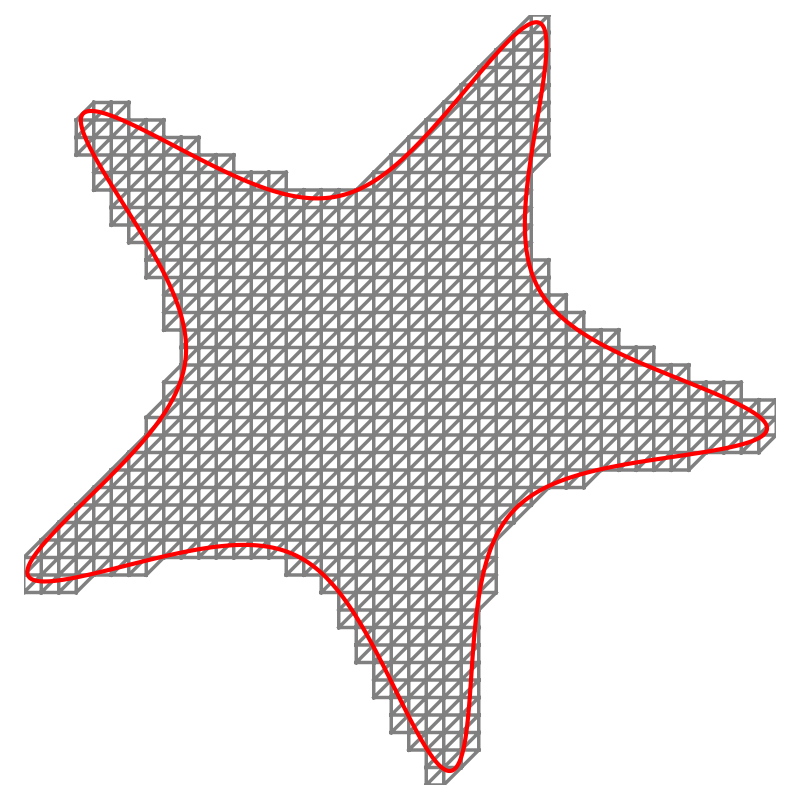}
    \hfill
    \includegraphics[width=0.32\linewidth]{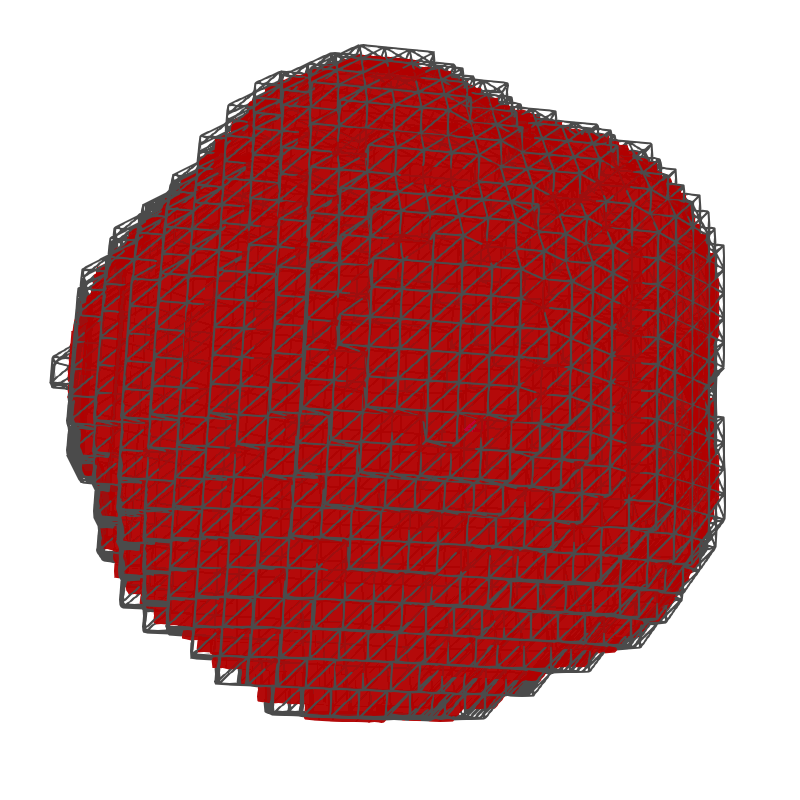}

    \caption{Overview of the three test cases. The grey regions represent the computational submesh $\mathcal{T}_h^n$ used by the $\varphi$-FEM method, while the red curves indicate the exact domain boundary $\partial\Omega(t_n)$. Left: translating circle. Middle: rotating star. Right: moving popcorn-shaped domain in 3D.}
    \label{fig:geometry}
\end{figure}

\noindent Before presenting the results, let us provide some implementation details. First, as explained in Section~\ref{Sect:Scheme_and_main_results}, for the discretization in time, we introduce the uniform grid $t_n = n \tau$ with $n=0,\dots,N$ and $\tau$ the time step. Hence, in order to implement the scheme \eqref{Eq:scheme}, we first define the sub-mesh $\mathcal{T}_h^n$ of the background mesh on which we work at each time step. Indeed, the theoretical definition of this sub-mesh requires evaluating the level-set for all $t \in I_n = [t_{n-1},t_n]$; however, this is not computationally feasible. We therefore choose to rewrite the definition of $\mathcal{T}_h^n$ as follows:
\[ \mathcal{T}_h^n = \left\{ T \in \mathcal{T}_h^{\mathcal{O}} : T \cap \Omega(t_n) \neq \varnothing \text{ or } T \cap \Omega(t_{n-1}) \neq \varnothing \right\}. \]
However, this new definition needs to be modified once again because it is not possible to determine exactly whether a cell is intersected by the boundary at every point. Therefore, the sub-mesh definition we implement is
\[ \mathcal{T}_h^n = \left\{ T \in \mathcal{T}_h^{\mathcal{O}} : \varphi(t_n,x) < 0 \text{ or } \varphi(t_{n-1},x) < 0, \text{ for at least one vertex $x$ of $T$} \right\}. \]

\begin{remark}
In the analysis given in Sections \ref{Sect:Scheme_and_main_results}--\ref{Sect:Proof}, the exact level-set function $\varphi$ is used in the discrete forms. In the implementation, $\varphi$ is replaced by its $\mathbb{P}_\ell$ finite element interpolation on the background mesh, with $\ell=1$ or $\ell=2$, as is customary in previous $\varphi$-FEM works. This substitution is not innocuous: as reported below for the first test case, the accuracy of the scheme is markedly better with $\ell=2$ than with $\ell=1$, which is why we use $\ell=2$ in the second and third test cases. Extending the analysis to a discrete level set is left for future work.
\end{remark}

Figure~\ref{fig:marker_cells} shows the sub-mesh $\mathcal{T}_h^n$ for the first test case. In this case, it is clear that the sub-mesh contains all the cells located completely in $\Omega(t_{n-1})$ or $\Omega(t_n)$, as well as those cut by the boundary of one of these two domains.

\noindent Furthermore, in the scheme \eqref{Eq:scheme}, we need to define several subdomains on which we integrate. Figure~\ref{fig:marker_cells} illustrates the numerical cell selection used to perform integration on these different subdomains. According to the previous definition of $\mathcal{T}_h^n$, we have
\begin{align*}
    \mathcal{T}_h^n \setminus \mathcal{T}_h^{n-1} = \biggl\{ T \in \mathcal{T}_h^\mathcal{O} : &\; \biggl[ \varphi(t_n, x) < 0 \text{ for at least one vertex } x \text{ of } T \biggr] \\
    &\; \text{and } \biggl[ \varphi(t_{n-1}, x) \geq 0 \text{ for all vertices } x \text{ of } T \biggr] \\
    &\; \text{and } \biggl[ \varphi(t_{n-2}, x) \geq 0 \text{ for all vertices } x \text{ of } T \biggr] \biggr\}.
\end{align*}

\begin{remark}
Note that $\Omega_h^0$ is the domain covered by
\[ \mathcal{T}_h^0 = \{ T \in \mathcal{T}_h^\mathcal{O} : \varphi(t_0, x) < 0 \text{ for at least one vertex } x \text{ of } T \}. \]
Hence, $\Omega_h^1 \cap \Omega_h^0 = \Omega_h^0$. In particular, for $n=1$ the formula above for $\mathcal{T}_h^n\setminus\mathcal{T}_h^{n-1}$ has to be understood without the condition involving $\varphi(t_{n-2},\cdot)$, which is not defined in that case.
\end{remark}

For the stabilization terms, we also need to define $\mathcal{T}_h^{n,\Gamma}$. Here again, the theoretical definition, which involves $\partial\Omega(t)$ for all $t\in I_n$, is replaced by a computable one, namely the set of cells of $\mathcal{T}_h^n$ that are not completely inside $\Omega(t_n) \cap \Omega(t_{n-1})$:
\[ \mathcal{T}_h^{n,\Gamma} = \left\{ T \in \mathcal{T}_h^n : \varphi(t_n,x) \geq 0 \text{ or } \varphi(t_{n-1},x) \geq 0 \text{ for at least one vertex } x \text{ of } T \right\}. \]
Finally, we define $\mathcal{F}_h^{n,\Gamma}$ as the set of faces of $\mathcal{T}_h^{n,\Gamma}$ that are not located on the boundary of $\mathcal{T}_h^n$.

\begin{figure}
    \centering
    \includegraphics[width=0.5\linewidth]{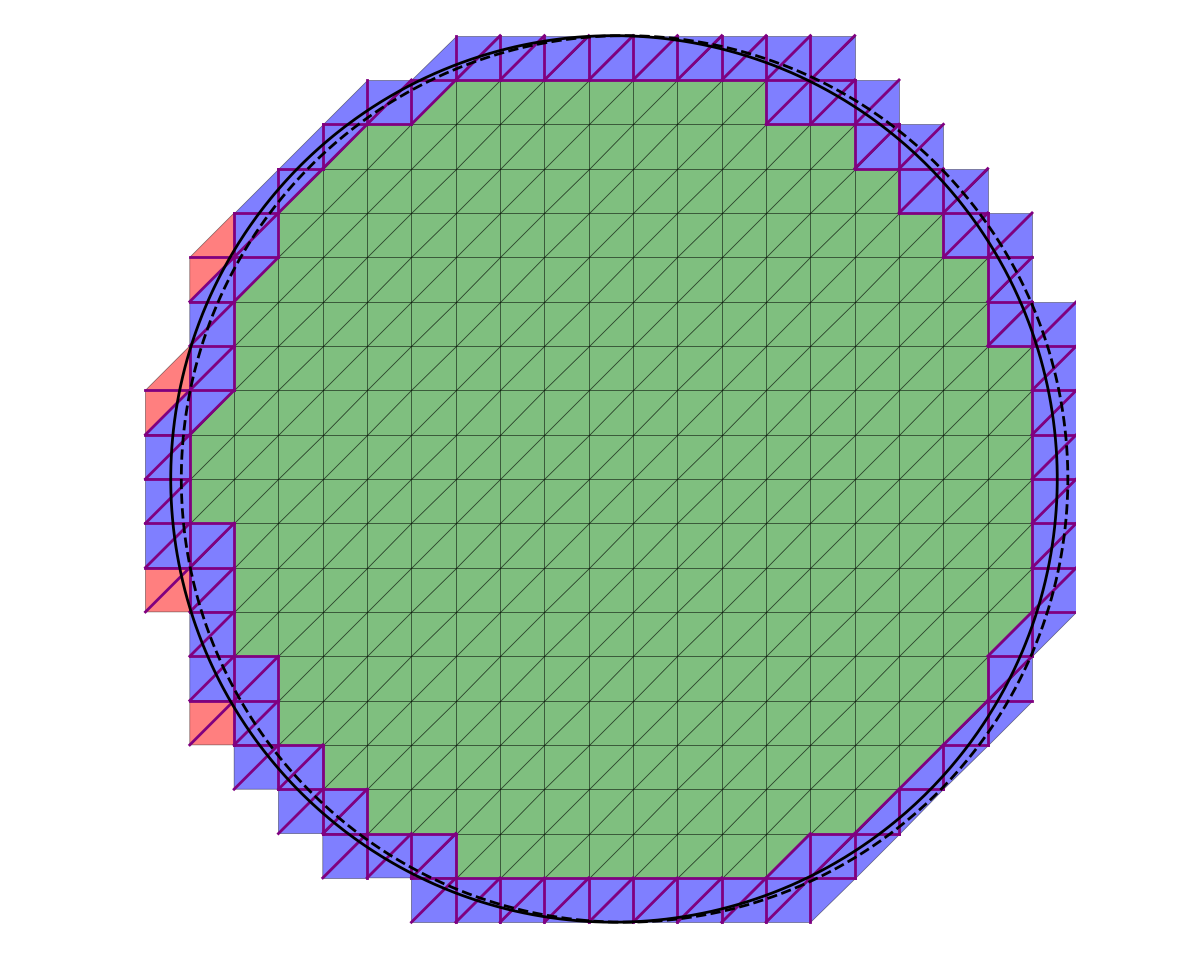}
    \caption{Submeshes for the first test case. The solid and dashed black lines represent $\varphi(t_n)$ and $\varphi(t_{n-1})$, respectively. The cell coloring denotes the following subsets: $\mathcal{T}_h^n$ (green, blue, red), $\mathcal{T}_h^{n-1} \cap \mathcal{T}_h^n$ (green, blue), $\mathcal{T}_h^n \setminus \mathcal{T}_h^{n-1}$ (red), and $\mathcal{T}_h^{n,\Gamma}$ (blue, red). The purple edges correspond to $\mathcal{F}_h^{n,\Gamma}$. For visualization purposes, the time step $\tau$ is intentionally chosen larger than in the actual simulations in order to clearly distinguish the different subdomains.}
    \label{fig:marker_cells}
\end{figure}

\noindent To continue with the implementation details of the scheme \eqref{Eq:scheme}, we use the midpoint quadrature rule for the time integration, so that for any continuous function $g$,
\[ \int_{I_n} g(t) \, \dt \approx \tau g(t_{n-1/2}), \]
with $t_{n-1/2} = (t_{n-1} + t_n)/2$.

\noindent To quantify the accuracy of the method, we evaluate the following relative error norms:
\begin{subequations}\label{Eq:error}
    \begin{equation}
        \frac{\|u_h-u_\mathrm{ref}\|_{\ell^2(0,T;H^1(\Omega_\mathrm{ref}))}}{\|u_\mathrm{ref}\|_{\ell^2(0,T;H^1(\Omega_\mathrm{ref}))}} \approx \biggl( \frac{\sum_{n=1}^N \int_{\Omega_\mathrm{ref}^n} |\nabla u_h^n(x) - \nabla u_\mathrm{ref}(x, t_n)|^2 \ \dx }{\sum_{n=1}^N \int_{\Omega_\mathrm{ref}^n} |\nabla u_\mathrm{ref}(x, t_n)|^2 \ \dx} \biggr)^{1/2},
    \end{equation}
    \begin{equation}
        \frac{\|u_h-u_\mathrm{ref}\|_{\ell^\infty(0,T;L^2(\Omega_\mathrm{ref}))}}{\|u_\mathrm{ref}\|_{\ell^\infty(0,T;L^2(\Omega_\mathrm{ref}))}} \approx \biggl( \frac{\max_{1 \leq n \leq N} \int_{\Omega_\mathrm{ref}^n} | u_h^n(x) - u_\mathrm{ref}(x, t_n)|^2 \ \dx }{\max_{1 \leq n \leq N}\int_{\Omega_\mathrm{ref}^n} | u_\mathrm{ref}(x, t_n)|^2 \ \dx} \biggr)^{1/2}.
    \end{equation}
\end{subequations}
The reference domain $\Omega_{\mathrm{ref}}^n$ is a very fine mesh refinement of $\Omega_h^n$. On this refined submesh, $u_\mathrm{ref}$ denotes the exact (reference) solution, while $u_h$ represents the $\varphi$-FEM solution interpolated onto $\Omega_\mathrm{ref}^n$. Note that the sums and maxima start at $n=1$: the scheme \eqref{Eq:scheme} never introduces $\tilde u_h^0$, the initial condition being taken into account directly in the first time step, cf. Section \ref{Sect:Scheme_and_main_results}. Note also that the quantity in the first equation only involves the $H^1$ semi-norm; since $u_h^n$ and $u_\mathrm{ref}$ both vanish on the boundary, this is equivalent to the full $H^1_0$ norm.

Note that, if the expected convergence is of order $C_1 h^p + C_2 \tau^m$, then we will
fix $\tau = h^{p/m}$ in such a way {that} we only need to observe numerically if the error is of order $h^p$.

\subsection{First test case: a translated circle}

For the first test case, we consider a simple moving circle domain, illustrated in Figure~\ref{fig:geometry}, defined by the level-set function
\begin{equation}
    \varphi(x,y;t) = (x - \cos(\omega t))^2 + y^2 - 1,
    \label{eq:phi disk}
\end{equation}
where $\omega = \pi/8$. The domain is thus a disk of radius $1$ with a center oscillating between $x=-1$ and $x=1$, with a maximal speed $\omega$. The value of $\omega$ is chosen to be sufficiently small so that the interface travels at most one cell layer per time step, in agreement with Assumption \ref{asm2}; the same choice is made in the two other test cases.

\paragraph{Manufactured solution.} We first consider the exact solution
\begin{equation}
    u_{\mathrm{ref}}(x,y;t)
    =
    \varphi(x,y;t)\sin(x)\sin(y),
    \label{Eq:manufactured solution}
\end{equation}
which satisfies the homogeneous Dirichlet boundary condition on $\partial\Omega(t)$ since $\varphi=0$ on the boundary implies $u_{\mathrm{ref}}=0$. Using this exact solution, we evaluate the relative errors in the $\ell^2(0,T;H^1(\Omega_{\mathrm{ref}}))$ and $\ell^\infty(0,T;L^2(\Omega_{\mathrm{ref}}))$ norms. The computations are performed with $\mathbb{P}_1$ finite elements and the convergence results are reported in Figure~\ref{fig:convergence1}.

To assess the performance of the $\varphi$-FEM scheme, we compare it against the classical conforming FEM method. However, it is important to note that the standard FEM cannot be straightforwardly implemented for moving domain problems---or at least, a naive implementation is not feasible. Indeed, when the domain evolves over time, the mesh nodes move, and the key difficulty lies in computing the material time derivative at points that did not belong to the mesh at the previous time step. To avoid this issue and enable a fair comparison, we resort to a classical change of variables. Starting from the heat equation formulated on the moving domain \eqref{eq:heat}, we pull back the equation onto a fixed reference domain. In our specific setting, the moving domain is defined by the level-set function \eqref{eq:phi disk} which describes a unit disk whose center translates along the $x$-axis. The reference domain is the unit disk $\hat{\Omega} = \{(\xi,\eta): \xi^2+\eta^2<1\}$, and we define $\hat{u}(\xi,\eta;t) = u(\xi + \cos(\omega t), \eta; t)$. Denoting by $\mathbf{w}(t) = (-\omega \sin(\omega t), 0)$ the velocity of the domain, the chain rule gives
\[
\partial_t u = \partial_t \hat{u} - \mathbf{w} \cdot \nabla \hat{u}, \qquad \Delta u = \Delta \hat{u}.
\]
The transformed equation on the fixed domain $\hat{\Omega}$ then reads
\[
\partial_t \hat{u} - \Delta \hat{u} - \mathbf{w}(t) \cdot \nabla \hat{u} = \hat{f},
\]
with $\hat{f}(\xi,\eta;t) = f(\xi + \cos(\omega t), \eta; t)$, supplemented with homogeneous Dirichlet conditions on $\partial\hat{\Omega}$ and the initial condition $\hat{u}(\xi,\eta;0) = {u^0}(\xi + 1, \eta)$. This formulation is now set on a fixed mesh, for which the conforming FEM applies without ambiguity.

The numerical results show that the $\varphi$-FEM scheme is sensitive to the interpolation order used for the level-set function $\varphi$. With $\mathbb{P}_1$ interpolation of $\varphi$, its accuracy is below that of the conforming FEM, although the convergence rates are the same. Using $\mathbb{P}_2$ interpolation of $\varphi$ improves the accuracy substantially, and the resulting $\ell^\infty(L^2)$ error becomes even smaller than that of the conforming FEM. For this reason, $\mathbb{P}_2$ interpolation of the level set is used in the two remaining test cases.

\begin{figure}[h]
    \centering
    \begin{tikzpicture}
        \begin{loglogaxis}[width=0.7\textwidth, xlabel={$h$}, ylabel={Error}, legend to name=Legend, legend columns=2]
        
            \addplot[mark=*, color=red, dashed] table[x=h, y=L2] {\FigFourPOneData};
            \addlegendentry{$\ell^\infty(0,T;L^2)$ $\varphi$-FEM ($\mathbb{P}_1$ interpolation)}
            \addplot[mark=*, color=blue, dashed] table[x=h, y=H1] {\FigFourPOneData};
            \addlegendentry{$\ell^2(0,T;H^1)$ $\varphi$-FEM ($\mathbb{P}_1$ interpolation)}
            
            \addplot[mark=*, color=red, dotted] table[x=h, y=L2] {\FigFourPTwoData};
            \addlegendentry{$\ell^\infty(0,T;L^2)$ $\varphi$-FEM ($\mathbb{P}_2$ interpolation)}
            \addplot[mark=*, color=blue, dotted] table[x=h, y=H1] {\FigFourPTwoData};
            \addlegendentry{$\ell^2(0,T;H^1)$ $\varphi$-FEM ($\mathbb{P}_2$ interpolation)}
           
            \addplot[mark=*, color=red] table[x=h_squared, y=L2(FEM)] {\FigFourFEMData};
            \addlegendentry{$\ell^\infty(0,T;L^2)$ FEM}
            \addplot[mark=*, color=blue] table[x=h_squared, y=H1(FEM)] {\FigFourFEMData};
            \addlegendentry{$\ell^2(0,T;H^1)$ FEM}
           
            \logLogSlopeTriangle{0.4}{0.2}{0.5}{1}{black} 
            \logLogSlopeTriangle{0.6}{0.2}{0.25}{2}{black}
        \end{loglogaxis}
    \end{tikzpicture}
    
    \begin{center}
        \pgfplotslegendfromname{Legend}
    \end{center}

    \caption{Convergence results for the translating circle using a manufactured solution. Relative errors are plotted against the mesh size $h$ for $\mathbb{P}_1$ finite elements with $\gamma=1$. The colour encodes the norm: blue for $\ell^2(0,T;H^1)$, computed with $\Delta t = h$, and red for $\ell^\infty(0,T;L^2)$, computed with $\Delta t = h^2$, following the rule $\tau=h^{p/m}$ stated above. The line style encodes the scheme: solid lines, conforming FEM on the fixed reference domain $\hat\Omega$; dashed lines, $\varphi$-FEM with $\mathbb{P}_1$ interpolation of $\varphi$; dotted lines, $\varphi$-FEM with $\mathbb{P}_2$ interpolation of $\varphi$. All errors are evaluated on the reference mesh $\Omega_\mathrm{ref}^n$. The reference slopes $O(h)$ and $O(h^2)$ are indicated for comparison. First-order convergence is observed for the $\ell^2(H^1)$ norm, in agreement with Theorem \ref{thm:error}, while the $\ell^\infty(L^2)$ norm achieves second-order convergence, which is higher than the rate $k+1/2=3/2$ conjectured after Theorem \ref{thm:error}.}
    \label{fig:convergence1}
\end{figure}
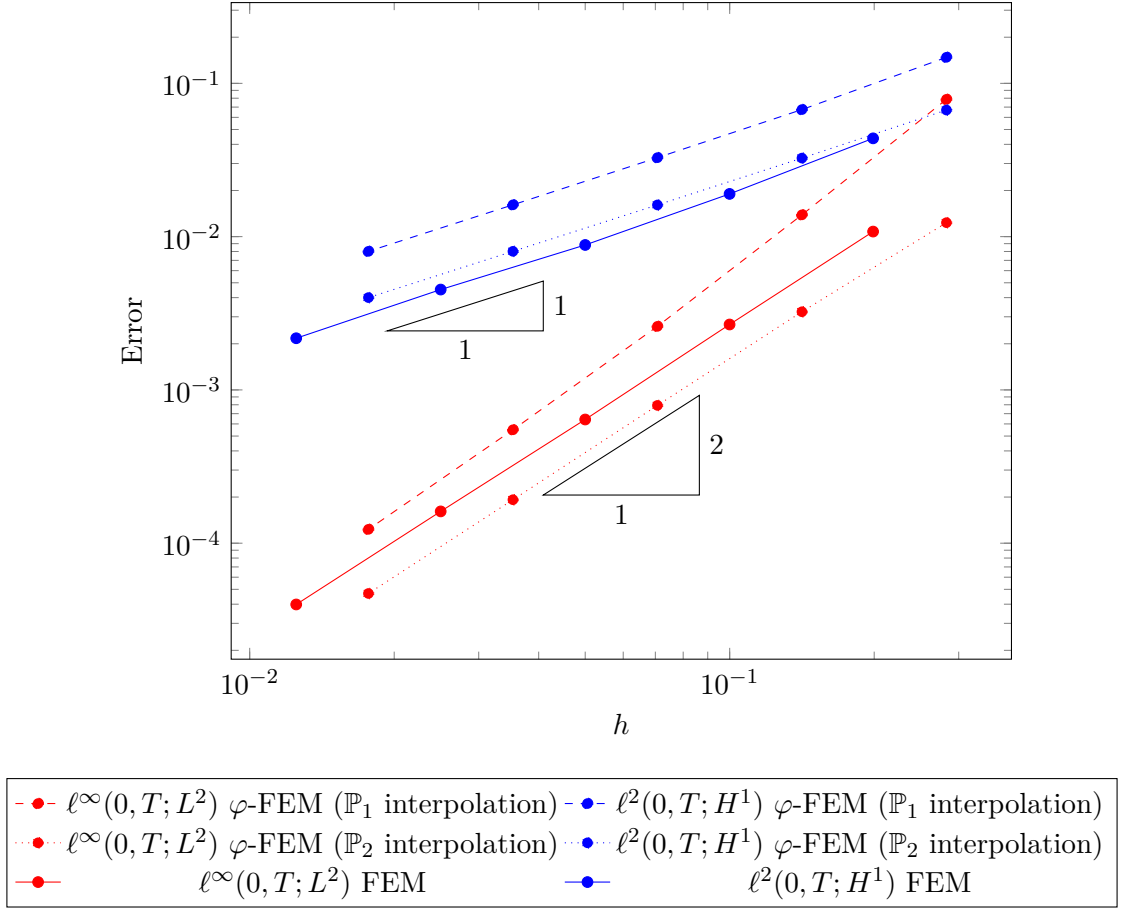

\paragraph{Influence of the stabilization parameter.} We investigate the sensitivity of the method to the stabilization parameter $\gamma$. Figure \ref{fig:gamma} displays the relative errors for several values of $\gamma$ and different mesh sizes. The results highlight the crucial role played by this stabilization parameter. The optimal choice yielding  the best precision appears to be $\gamma = 1$.

\begin{figure}
    \centering
    \begin{tikzpicture}
        \begin{loglogaxis}[width=0.48\textwidth, xlabel={$h$}, ylabel={Error $\ell^2(H^1)$}, legend to name=CommonLegend, legend columns=5]
            \addplot[thick, mark=*, color=blue] table[x=h, y=rel_L2_H1, col sep=space, restrict expr to domain={\thisrow{gamma}}{0.09:0.11}]{\FigFiveData};
            \addlegendentry{$\gamma = 10^{-1}$}
            
            \addplot[thick, mark=*, color=red] table[x=h, y=rel_L2_H1, col sep=space, restrict expr to domain={\thisrow{gamma}}{0.99:1.01}]{\FigFiveData};
            \addlegendentry{$\gamma = 1$}
            
            \addplot[thick, mark=*, color=green!60!black] table[x=h, y=rel_L2_H1, col sep=space, restrict expr to domain={\thisrow{gamma}}{9.9:10.1}]{\FigFiveData};
            \addlegendentry{$\gamma = 10$}
            
            \addplot[thick, mark=*, color=orange] table[x=h, y=rel_L2_H1, col sep=space, restrict expr to domain={\thisrow{gamma}}{99:101}]{\FigFiveData};
            \addlegendentry{$\gamma = 10^{2}$}
            
            \addplot[thick, mark=*, color=purple] table[x=h, y=rel_L2_H1, col sep=space, restrict expr to domain={\thisrow{gamma}}{999:1001}]{\FigFiveData};
            \addlegendentry{$\gamma = 10^{3}$}
            
            \logLogSlopeTriangle{0.75}{0.3}{0.2}{1}{black}
        \end{loglogaxis}
    \end{tikzpicture}
    \begin{tikzpicture}
        \begin{loglogaxis}[width=0.48\textwidth, xlabel={$h$}, ylabel={Error $\ell^\infty(L^2)$}]
            \addplot[thick, mark=*, color=blue, forget plot] table[x=h, y=rel_Linf_L2, col sep=space, restrict expr to domain={\thisrow{gamma}}{0.09:0.11}]{\FigFiveData};
            
            \addplot[thick, mark=*, color=red, forget plot] table[x=h, y=rel_Linf_L2, col sep=space, restrict expr to domain={\thisrow{gamma}}{0.99:1.01}]{\FigFiveData};
            
            \addplot[thick, mark=*, color=green!60!black, forget plot] table[x=h, y=rel_Linf_L2, col sep=space, restrict expr to domain={\thisrow{gamma}}{9.9:10.1}]{\FigFiveData};
            
            \addplot[thick, mark=*, color=orange, forget plot] table[x=h, y=rel_Linf_L2, col sep=space, restrict expr to domain={\thisrow{gamma}}{99:101}]{\FigFiveData};
            
            \addplot[thick, mark=*, color=purple, forget plot] table[x=h, y=rel_Linf_L2, col sep=space, restrict expr to domain={\thisrow{gamma}}{999:1001}]{\FigFiveData};
            
            \logLogSlopeTriangle{0.75}{0.3}{0.2}{2}{black}
        \end{loglogaxis}
    \end{tikzpicture}

    \begin{center}
        \pgfplotslegendfromname{CommonLegend}
    \end{center}

    \caption{First test case, translating circle. Influence of the stabilization parameter $\gamma$ on the convergence. Left: $\ell^2(H^1)$ relative error. Right: $\ell^\infty(L^2)$ relative error. Results are shown for $\Delta t = h$ and various $\gamma$. The value $\gamma = 1$ provides the best compromise between stability and accuracy.}
    \label{fig:gamma}
\end{figure}
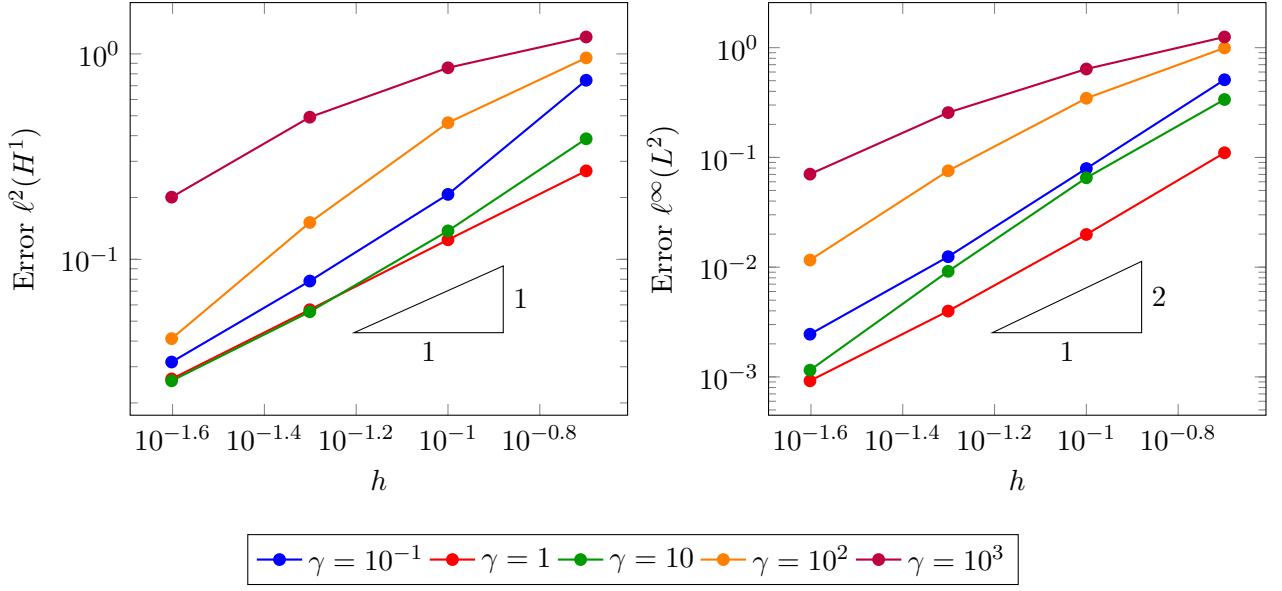

\paragraph{\texorpdfstring{Manufactured solution with space-time $\tilde{u}$.}{Manufactured solution with space-time u-tilde.}}
We now consider the exact solution
\begin{equation}
    u_{\mathrm{ref}}(x,y;t)
    =
    \varphi(x,y;t)\sin(x)\sin(y)e^{-t},
    \label{Eq:manufactured solution time-dependent}
\end{equation}
which satisfies the homogeneous Dirichlet boundary condition on $\partial\Omega(t)$ since $\varphi=0$ on the boundary implies $u_{\mathrm{ref}}=0$. In contrast to the previous case, the profile $\tilde{u}(x,y;t) = \sin(x)\sin(y)e^{-t}$ now depends on time as well as space, so that the time dependence of the exact solution is no longer carried solely by $\varphi$. Using this exact solution, we evaluate the relative errors in the $\ell^2(0,T;H^1(\Omega_{\mathrm{ref}}))$ and $\ell^\infty(0,T;L^2(\Omega_{\mathrm{ref}}))$ norms. The computations are performed with $\mathbb{P}_1$ finite elements and the convergence results are reported in Figure~\ref{fig:convergence1bis}.

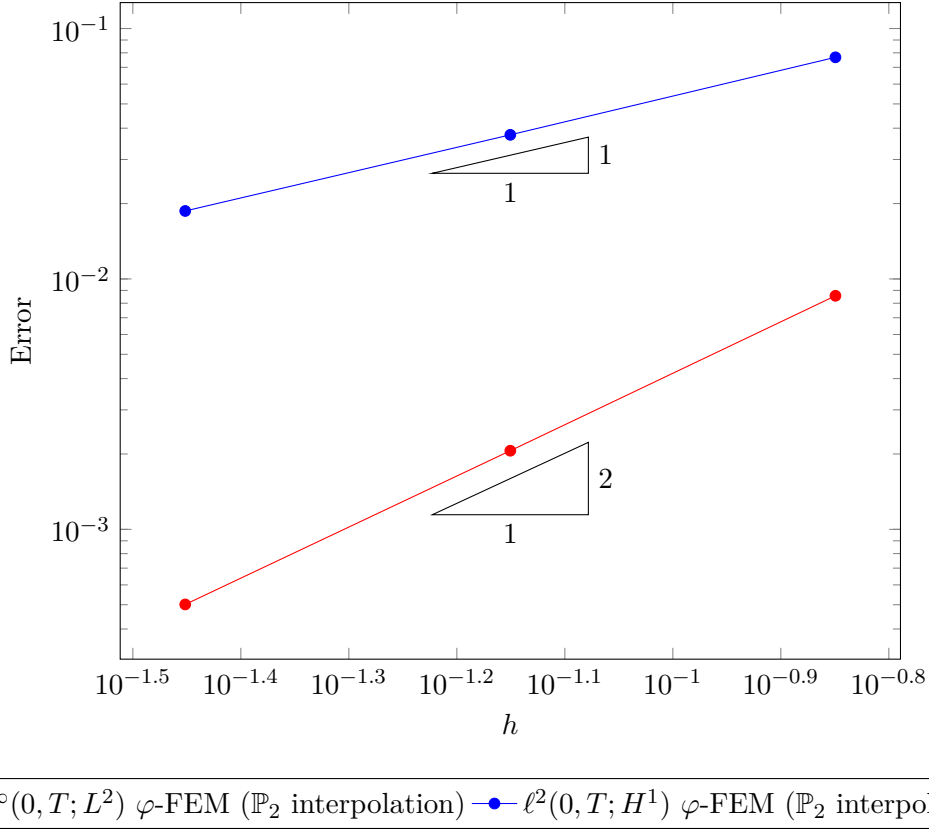
\begin{figure}[h]
    \centering
    \begin{tikzpicture}
        \begin{loglogaxis}[width=0.7\textwidth, xlabel={$h$}, ylabel={Error}, legend to name=LegendFigSpaceTime, legend columns=2]
        
            \addplot[mark=*, color=red] table[x=h, y=L2] {\FigSixData};
            \addlegendentry{$\ell^\infty(0,T;L^2)$ $\varphi$-FEM ($\mathbb{P}_2$ interpolation)}
            \addplot[mark=*, color=blue] table[x=h, y=H1] {\FigSixData};
            \addlegendentry{$\ell^2(0,T;H^1)$ $\varphi$-FEM ($\mathbb{P}_2$ interpolation)}
           
            \logLogSlopeTriangle{0.6}{0.2}{0.74}{1}{black} 
            \logLogSlopeTriangle{0.6}{0.2}{0.22}{2}{black}
        \end{loglogaxis}
    \end{tikzpicture}
    
    \begin{center}
        \pgfplotslegendfromname{LegendFigSpaceTime}
    \end{center}
    
    \caption{First test case, translating circle. Convergence results using a manufactured solution with time-space $\tilde{u}$, with $\tilde{u}$ depending on space and time. Relative errors in the $\ell^2(H^1)$ and $\ell^\infty(L^2)$ norms are plotted against $h$ for $\mathbb{P}_1$ elements, $\mathbb{P}_2$ interpolation of the level set and $\gamma=1$. Blue: the $\ell^2(H^1)$ error, computed with $\Delta t = h$; red: the $\ell^\infty(L^2)$ error, computed with $\Delta t = h^2$. The observed convergence rates are first-order for the $\ell^2(H^1)$ norm and second-order for the $\ell^\infty(L^2)$ norm. As for the first test case, we have a good agreement with the theoretical result for the $\ell^2(H^1)$ relative error norms, while we obtain better accuracy in norm $\ell^\infty(L^2)$.}
    \label{fig:convergence1bis}
\end{figure}

\subsection{Second test case: a star in rotation}

\noindent For this second test case, the source term is derived from a manufactured solution, allowing us to compare the $\varphi$-FEM solution with the exact one. We consider a rotating ``star'' centered at the origin, see Figure~\ref{fig:geometry}, described by the following level-set function:
\begin{equation*}
    \varphi(x,y;t) = (x^2 + y^2)^2 \bigl(2.5 + 2.25 \sin(5(\operatorname{atan2}(y, x) - \omega t))\bigr) - 1,
\end{equation*}
where $\omega = \pi/8$ is the angular velocity. {This} geometry is inspired by the first test case considered in \cite{neumann}. In this domain, we consider the exact solution $u_{\mathrm{ref}}(x,y;t)=\varphi(x,y;t)\sin(x)\cos(y)$. Figure \ref{fig:convergence2} demonstrates that the method maintains its optimal convergence properties even for {geometries with a strongly varying curvature}. 

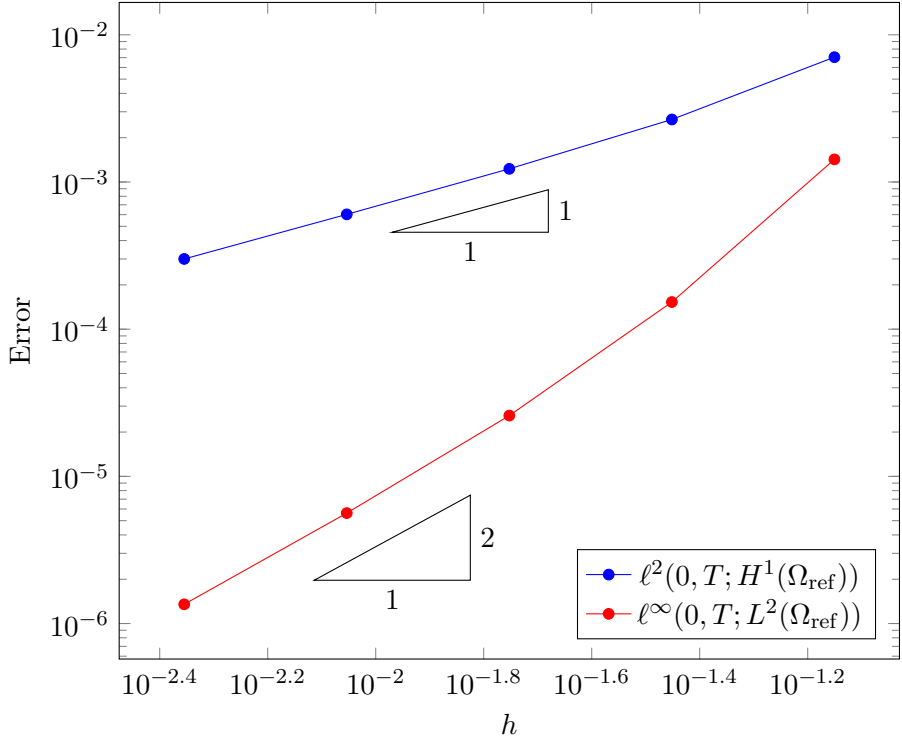
\begin{figure}
    \centering
    \begin{tikzpicture}
        \begin{loglogaxis}[width=0.7\textwidth, xlabel={$h$}, ylabel={Error}, legend pos=south east]
            \addplot[mark=*, color=blue] table[x=h, y=H1] {\FigSevenData};
            \addlegendentry{$\ell^2(0,T;H^1(\Omega_{\mathrm{ref}}))$}
            
            \addplot[mark=*, color=red] table[x=h, y=L2] {\FigSevenData};
            \addlegendentry{$\ell^\infty(0,T;L^2(\Omega_{\mathrm{ref}}))$}
            
            \logLogSlopeTriangle{0.55}{0.2}{0.65}{1}{black} 
            \logLogSlopeTriangle{0.45}{0.2}{0.12}{2}{black}
        \end{loglogaxis}
    \end{tikzpicture}

    \caption{Second test case, rotating star. Convergence results using a manufactured solution. Relative errors in the $\ell^2(H^1)$ and $\ell^\infty(L^2)$ norms are plotted against $h$ for $\mathbb{P}_1$ elements, $\mathbb{P}_2$ interpolation of the level set and $\gamma=1$. Blue: the $\ell^2(H^1)$ error, computed with $\Delta t = h$; red: the $\ell^\infty(L^2)$ error, computed with $\Delta t = h^2$. The observed convergence rates are first-order for the $\ell^2(H^1)$ norm and second-order for the $\ell^\infty(L^2)$ norm. As for the first test case, we have a good agreement with the theoretical result for the $\ell^2(H^1)$ relative error norms, while we obtain better accuracy in norm $\ell^\infty(L^2)$.}
    \label{fig:convergence2}
\end{figure}

\subsection{Third test case: a 3D moving {popcorn}}

\noindent Finally, we assess the performance of the method in three dimensions on a moving popcorn-shaped domain, see Figure \ref{fig:geometry}. The level-set function is defined as:
\begin{equation*}
    \varphi(x,y,z;t) = x_{\text{rot}}^2 + y_{\text{rot}}^2 + z^2 - (0.6)^2 - 1.5 \sum_{k=0}^{11} \exp \left(-\frac{(x_{\text{rot}} - x_k)^2 + (y_{\text{rot}} - y_k)^2 + (z - z_k)^2}{0.3^2} \right),
\end{equation*}
with
\begin{align*}
    &x_{\text{rot}} = x \cos(\omega t) - y \sin(\omega t), \quad 
    y_{\text{rot}} = x \sin(\omega t) + y \cos(\omega t),\\
    &x_k = \frac{1.2}{\sqrt{5}} \cos\left(\frac{2 k \pi}{5}\right),\quad
    y_k = \frac{1.2}{\sqrt{5}} \sin\left(\frac{2 k \pi}{5}\right),\quad
    z_k = \frac{0.6}{\sqrt{5}}, \qquad k=0,\dots,4;\\
    &x_k = \frac{1.2}{\sqrt{5}} \cos\left(\frac{(2k - 11)\pi}{5}\right),\quad
    y_k = \frac{1.2}{\sqrt{5}} \sin\left(\frac{(2k - 11)\pi}{5}\right),\quad
    z_k = -\frac{0.6}{\sqrt{5}}, \qquad k=5,\dots,9;\\
    &x_k = 0,\ y_k = 0,\ z_k = 0.6, \qquad k=10;\\
    &x_k = 0,\ y_k = 0,\ z_k = -0.6, \qquad k=11;\\
    &\omega = \frac{\pi}{8} \quad \text{(angular velocity)}.
\end{align*}
The points $(x_k,y_k,z_k)$ define the positions of the popcorn bumps. This geometry represents a challenging test case due to its fine {geometrical} features. It is inspired by the third test case considered in \cite{heat}.
In this domain, we consider the exact solution $u_{\mathrm{ref}}(x,y,z;t)=\varphi(x,y,z;t)\sin(x)\sin(y)\sin(z)$.

Figure \ref{fig:convergence3} shows that the optimal convergence orders are preserved in three dimensions. This result confirms the robustness and scalability of the $\varphi$-FEM approach for problems involving complex, time-dependent domains in higher dimensions.

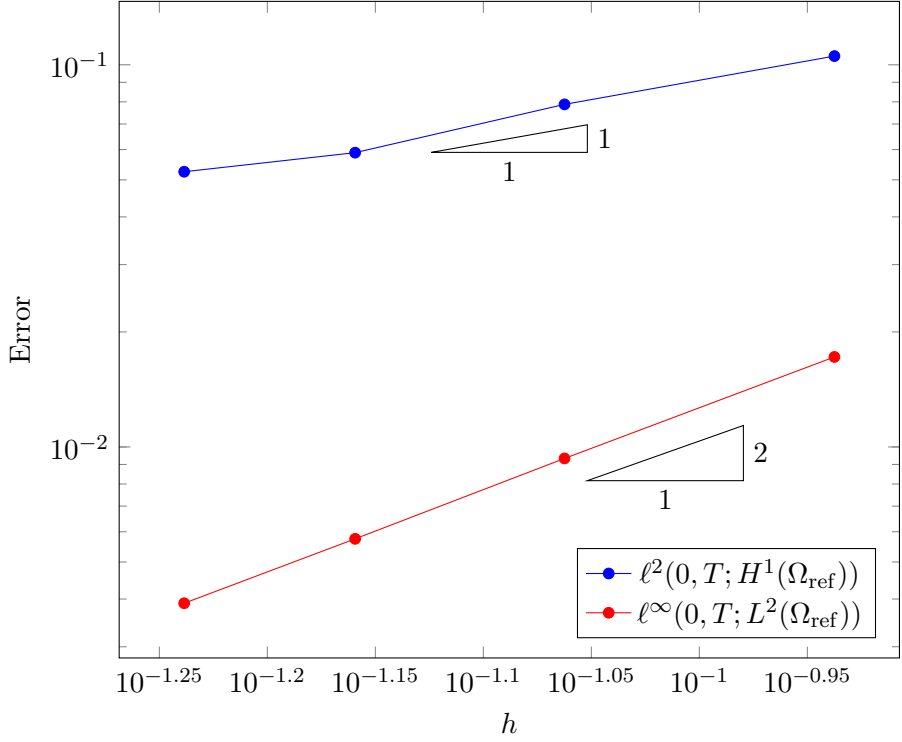
\begin{figure}
    \centering
    \begin{tikzpicture}
        \begin{loglogaxis}[width=0.7\textwidth, xlabel={$h$}, ylabel={Error}, legend pos=south east]
        
        \addplot[mark=*, color=blue] table[x=h, y=H1] {\FigEightData};
        \addlegendentry{$\ell^2(0,T;H^1(\Omega_{\mathrm{ref}}))$}
        
        \addplot[mark=*, color=red] table[x=h, y=L2] {\FigEightData};
        \addlegendentry{$\ell^\infty(0,T;L^2(\Omega_{\mathrm{ref}}))$}
        
        \logLogSlopeTriangle{0.6}{0.2}{0.77}{1}{black} 
        \logLogSlopeTriangle{0.8}{0.2}{0.27}{2}{black}
        \end{loglogaxis}
    \end{tikzpicture}
    
    \caption{Third test case, a moving {popcorn} in 3D. Relative errors with respect to $h$ for $\mathbb{P}_1$ elements, $\mathbb{P}_2$ interpolation of the level set and $\gamma=1$. Blue: the $\ell^2(H^1)$ error, computed with $\Delta t = h$; red: the $\ell^\infty(L^2)$ error, computed with $\Delta t = h^2$. The observed convergence rates are first-order in $\ell^2(H^1)$ and second-order in $\ell^\infty(L^2)$, demonstrating that the $\varphi$-FEM scheme is effective even for complex three-dimensional evolving geometries. As for the two-dimensional cases, we have a good agreement with the theoretical result for the $\ell^2(H^1)$ relative error norms, while we obtain better accuracy in norm $\ell^\infty(L^2)$.}
    \label{fig:convergence3}
\end{figure}

\section{Conclusions}
\label{Sect:Conclusion}

We have proposed and analysed a $\varphi$-FEM scheme for the heat equation posed on a domain evolving in time. The scheme combines the $\varphi$-FEM paradigm in space, in which the level-set function is inserted into the discrete unknown so that the homogeneous Dirichlet condition is satisfied by construction, with a lowest-order discontinuous Galerkin discretization in time on space--time slabs. The active mesh is rebuilt at each slab from the background mesh only, so that no remeshing and no mesh deformation is needed, and the implementation reduces to assembling standard bilinear forms on submeshes of a fixed background mesh.

Under regularity assumptions on the level set and a patch-covering assumption on the boundary submesh, we have proved a coercivity estimate for the stabilized bilinear form and derived an optimal \emph{a priori} error estimate of order $O(h^k+\tau)$ in the $L^2(0,T;H^1)$ norm. The three test cases -- a translating disk, a rotating star and a moving three-dimensional popcorn -- confirm this rate. They also exhibit a second-order convergence in the $\ell^\infty(0,T;L^2)$ norm, which is better than the rate $k+1/2$ one would expect from the analysis, and they show that the accuracy of the method depends significantly on the interpolation order used for the level set.

Several directions remain open. The most immediate one is the proof of the $L^{\infty}(L^2)$ estimate stated after Theorem \ref{thm:error}, which would require adapting a discrete Aubin--Nitsche duality argument to moving domains. A second one is to carry out the analysis directly with a discrete level-set function, as is done for the stationary problem, so as to cover the scheme that is actually implemented. Finally, higher-order elements in space, higher-order discontinuous Galerkin discretizations in time, and the extension to Neumann or interface conditions and to the Stokes system on moving domains are natural continuations of this work.

\section*{Declaration of generative AI use}
Claude (Anthropic) was used to improve the language and presentation of the manuscript, and to assist with the implementation of the algorithms of Section~\ref{Sect:Numerical_results}. All AI-assisted text was subsequently reviewed and edited by the authors. All code was reviewed by the authors and validated against analytical solutions and the convergence rates predicted by the theory. 

\section*{Funding}
This work was supported by the Agence Nationale de la Recherche, Project PhiFEM, under grant ANR-22-CE46-0003-01, and by the EIPHI Graduate School (contract ANR-17-EURE-0002).
This project has also received financial support from the CNRS through the MITI interdisciplinary programs.

\bibliographystyle{abbrvnat}
\bibliography{biblio}

\end{document}